\documentclass[twoside,reqno]{amsart}
\usepackage{amsmath}

\usepackage{amsfonts}
\usepackage{amsthm}
\usepackage{amssymb}
\usepackage{graphicx,psfrag,epsfig}
\usepackage{color}
\usepackage{mathrsfs}
\usepackage{pdfsync}
\usepackage{pst-all}
\usepackage{cite}
\usepackage{times}

\DeclareMathOperator*{\esssup}{ess-sup}

\newcommand{\R}{\mathbb{R}}

\newcommand{\N}{\mathbb{N}}

\newcommand{\rd}{\mathrm{d}}

\newcommand{\bqn}{\begin{equation}}
\newcommand{\eqn}{\end{equation}}
\newcommand{\bqnn}{\begin{equation*}}
\newcommand{\eqnn}{\end{equation*}}
\newcommand{\bear}{\begin{eqnarray}}
\newcommand{\eear}{\end{eqnarray}}
\newcommand{\bean}{\begin{eqnarray*}}
\newcommand{\eean}{\end{eqnarray*}}
\newcommand{\esxyp}{\end{split}}
\newcommand{\bsp}{\begin{split}}
\newcommand{\balg}{\begin{align}}
\newcommand{\ealg}{\end{align}}
\newcommand{\balgg}{\begin{align*}}
\newcommand{\ealgg}{\end{align*}}

\newtheorem{theorem}{Theorem}[section]
\newtheorem{corollary}[theorem]{Corollary}
\newtheorem{lemma}[theorem]{Lemma}
\newtheorem{proposition}[theorem]{Proposition}
\newtheorem{definition}[theorem]{Definition}

\numberwithin{equation}{section}
\begin{document}
\title{ Existence of Global Classical and Weak Solutions to \\ a Prion Equation with Polymer Joining}

\author{Elena Leis}
\address{Leibniz Universit\"at Hannover, Institut f\" ur Angewandte Mathematik, Welfengarten 1, D--30167 Hannover, Germany} 
\email{leis@ifam.uni-hannover.de}

\author{Christoph Walker}
\address{Leibniz Universit\"at Hannover, Institut f\" ur Angewandte Mathematik, Welfengarten 1, D--30167 Hannover, Germany} 
\email{walker@ifam.uni-hannover.de}
\keywords{Prions, polymer joining, classical and weak solutions, evolution operators.}
\subjclass{}
\date{\today}
%
\begin{abstract}
We consider a nonlinear integro-differential equation for prion proliferation that includes prion polymerization, polymer splitting, and polymer joining. The equation can be written as a quasilinear Cauchy problem. For bounded reaction rates we prove global existence  and uniqueness of classical solutions by means of evolution operator theory. We also prove global existence of  weak solutions for unbounded reaction rates by a compactness argument. 
\end{abstract}
%
\maketitle
\pagestyle{myheadings}
\markboth{\sc{E. Leis \& Ch. Walker}}{\sc{A Prion Equation with Polymer Joining}}
%
%
\section{Introduction} \label{sec:int}

Prions are misfolded proteins and are regarded as the infectious agent
of fatal diseases known as TSE's including BSE of cattle,
new variant Creutzfeldt-Jakob of human, and Scrapie of sheep.
Prions seem to be capable of proliferation despite  lacking DNA  and RNA. In this article we focus on  a mathematical model introduced in \cite{GvDWW07} for  {\it nucleated polymerization}
which is a theory describing the replication of prions. According to this theory,
infectious $PrP^{Sc}$ prions are thought to be a polymer form  of a normal protein monomer $PrP^{C}$. Infectious polymers 
build bonds involving several thousands of monomer units by
attaching non-infectious $PrP^{C}$ monomers and converting them to
the infectious form. Prions are very stable but can also
split into smaller polymers. Usually, this produces again two
infectious $PrP^{Sc}$ polymers. However, decay products below a critical size $y_0>0$ are assumed to disintegrate instantaneously  into $PrP^{C}$ monomers. Moreover, two infectious polymers can also join and form longer polymers. We refer to
\cite{GvDWW07,GPW06,Nowak, Masel} and the references therein for
more detailed information on the biological background and on the mechanism of nucleated polymerization.

The biological processes of polymerization, polymer joining,
and polymer splitting can be described by a coupled system consisting of
an ordinary differential equation for the number of $PrP^{C}$
monomers $v(t)\ge 0$ and an integro-differential equation for the
density distribution function $u=u(t,y)\ge 0$ for $PrP^{Sc}$
polymers of size $y>y_0$. The monomer equation is
\begin{equation}
\label{eqv}
\begin{split}
v'(t) & = \lambda-\gamma v(t) - \frac{v(t)}{1+\nu \displaystyle\int_{y_0}^{\infty}  u(t,z) z \mathrm{d}z} \int_{y_0}^{\infty} \tau(y)  u(t,y) \, \mathrm{d}y 
\\ & \quad + 2\int_{y_0}^{\infty} u(t,y) \beta (y) \int_0^{y_0} z \kappa(z,y)  \, \mathrm{d}z \, \mathrm{d}y   
\end{split}
\end{equation}
and the polymer equation is 
\begin{equation}\label{equ}
\begin{split}
 \partial_tu(t,y) +  \frac{v(t)}{1+\nu \displaystyle\int_{y_0}^{\infty}   u(t,z) z \mathrm{d}z}\partial_y{\left(\tau(y) u(t,y) \right)}  =   L[u(t)](y) + Q[u(t),u(t)](y)
\end{split}
\end{equation}
for $t>0$ and $y \in  Y:=(y_0, \infty)$ involving a linear part $L$ with
\begin{equation*}
L[u](y) :=  - (\mu(y) + \beta(y))  u(y)  +  2 \int_{y}^{\infty}  \beta(z)  \kappa(y,z)  u(z)    \, \mathrm{d}z
\end{equation*}
and a bilinear part $Q$ with 
\begin{equation*}
Q[u,w](y) := \mathbf{1}_{[y>2y_0]}\int_{y_0}^{y-y_0}\eta(y-z,z)u(y-z)w(z)\, \mathrm{d}z - 2u(y)\int_{y_0}^{\infty}  \eta(z,y)w(z) \, \mathrm{d}z \,.
\end{equation*}
The equations are supplemented with the boundary condition
\begin{equation}
\label{RB}
u(t,y_0) = 0 \ ,  \quad t>0  
\end{equation}
and the initial values
\begin{equation}
\label{AW}
v(0) = v^0 \ , \quad u(0,y) = u^0(y) \ , \quad y \in (y_0, \infty) \ .
\end{equation}      
According to the right-hand side of the ordinary differential equation \eqref{eqv}  the number of monomers is increased by a
constant background source $\lambda$ and if a $PrP^{Sc}$ polymer of
any size $y>y_0$ decays at a rate $\beta(y)$ into at least one
daughter polymer of size $z\le y_0$, which is assumed to
disintegrate instantaneously into monomers only. The probability (density)
for this event is denoted by $\kappa(z,y)$. The number of $PrP^{C}$ monomers decreases by metabolic
degradation with rate $\gamma$ and
if monomers are attached to a
$PrP^{Sc}$ polymer of size $y>y_0$ at rate $\tau(y)$. Accordingly, equation \eqref{equ} for $u$ involves a nonlinear polymerization term 
$$
\frac{v(t)}{1+\nu \displaystyle\int_{y_0}^{\infty}   u(t,z) z \mathrm{d}z}\partial_y(\tau(y) u(y))\,.
$$ 
If $\nu>0$ there is a saturation effect when the number $\int_{y_0}^{\infty}  u(t,z) z \mathrm{d}z$ of monomers within the infectious polymers becomes large resulting in less lengthening overall. The right-hand side of \eqref{equ} reflects that polymers
of size $y>y_0$  disappear  due to metabolic degradation with
rate $\mu(y)$, by splitting with rate $\beta(y)$, or if they join with another polymer. Also, polymers of size $y>y_0$ can be produced by the decay of
a larger polymer or if two smaller polymers join. Thus, equation \eqref{equ} is reminiscent of the continuous coagulation-fragmentation equation known from physics (see e.g. \cite{ELMP,L00} and the references therein).
\\

When polymer joining is neglected, that is, $\eta\equiv 0$, \eqref{eqv}-\eqref{AW} and variants thereof were investigated in \cite{EPW06, GPW06, PP-MWZ06, LW07, SW06, W}. More precisely, assuming that the kernels have the particular form
 \begin{equation}\label{5}
     \tau \equiv \text{const}\,,\quad  \mu \equiv \text{const}\,,\quad
\beta(y) =      \beta  y\,,\quad \kappa(z,y)  = \frac{1}{y}\,,
    \end{equation}
\eqref{eqv}-\eqref{equ} can be integrated and a closed system of ordinary differential equations for the unknowns $v$, $\int_Y u(t,y)  \mathrm{d} y$, and $\int_Y y
u(t,y)\mathrm{d} y$ can be obtained which possesses a unique global solution as shown in \cite{GPW06, PP-MWZ06} (for $\nu=0$). In these articles also stability of equilibria were studied. Note that in this case the solution $v$ to \eqref{eqv} is then determined and thus \eqref{eqv}-\eqref{equ} decouples leaving one with a non-local, but linear integro-differential equation for $u$ for which well-posedness and asymptotic stability of equilibria were shown in \cite{EPW06}. For $\eta\equiv 0$, well-posedness of global classical and weak solutions to the coupled system \eqref{eqv}-\eqref{AW} without assuming \eqref{5} was established in \cite{LW07, SW06, W}. 
Let us also point out that certain qualitative aspects of \eqref{eqv}-\eqref{equ} (still with $\eta\equiv 0$) were investigated e.g. in \cite{CLDDMP, CLODLMP, DG10, G15}.
The model with polymer joining was introduced in \cite{GvDWW07}. Assuming \eqref{5} and $\eta\equiv const$, equations \eqref{eqv}-\eqref{equ} can again be integrated to a  system of ordinary differential equations for which global well-posedness and stability of equilibria was studied in \cite{GvDWW07}.

The main contribution of this article is the inclusion of the bilinear polymer joining part $Q[u,u]$. We prove existence and uniqueness of global classical solutions as in \cite{SW06, W} and existence of global weak solutions  as in \cite{LW07}. Note that this does not seem to be straightforward since 
the linear part $L[u]$ can be considered as a perturbation of the first order polymerization term 
and thus, for $\eta\equiv 0$ (i.e. $Q\equiv 0$), equation \eqref{equ} is homogeneous and considerably simpler to handle, see \cite{SW06,W}. Including $Q$ requires additional arguments and the proofs -- in particular for classical solutions -- become more involved as we shall see later on (see the remarks at the end of Subsection~\ref{Sec3.1}).

\section{Main Results}

Throughout this article we assume that
\begin{equation}\label{lambda_gamma}
\nu, \lambda, \gamma \geq 0 \ .
\end{equation}
The splitting kernel $ \kappa \geq 0$ is a measurable function defined on $\mathcal{K}:=\{(z, y); y_0< y < \infty, 0 < z < y\}$ satisfying the symmetry condition
\begin{equation}
\kappa(z, y)=\kappa(y-z, y) \ , \quad (z, y) \in \mathcal{K} \ ,
\label{bin_split}
\end{equation}
and is normalized according to
\begin{equation}\label{mon_pres} 
2 \int_0^y z \kappa(z, y) \, \mathrm{d} z= y \ , \quad  \text{a.a. } y \in Y \ .
\end{equation}
Thus, splitting conserves the number of monomers and \eqref{bin_split}, \eqref{mon_pres} imply  
\begin{equation}\label{int_kappa}
\int_0^y \kappa (z, y) \, \mathrm{d} z= 1 \ , \quad  \text{a.a. } y \in Y \ .
\end{equation}
The polymer joining kernel $\eta$ is symmetric, that is,
\begin{equation}\label{etasym}
\eta(y,z) = \eta(z,y) \,,\quad  y,z \in Y\,.
\end{equation}
We then remark that  \eqref{etasym} (formally) implies  the identities
\begin{equation}
\label{L_Test}
\begin{split}
\int_{y_0}^{\infty}   \varphi(y) L[u](y) \, \mathrm{d} y  = &  - \int_{y_0}^{\infty}  \varphi(y)  \mu(y)  u(y) \, \mathrm{d} y 
\\ &  +  \int_{y_0}^{\infty}  u(y)  \beta(y)  \left(   -\varphi(y)  +  2 \int_{y_0}^{y}  \varphi(z)  \kappa(z,y)   \, \mathrm{d} z  \right)   \, \mathrm{d} y \ 
\end{split} 
\end{equation}
and  
\begin{equation}\label{tilde}
\int_{y_0}^{\infty} \varphi(y)  Q[u,u](y)  \, \mathrm{d}y  = \int_{y_0}^{\infty}  \int_{y_0}^{\infty}  (\varphi(y+z) - \varphi(y) - \varphi(z))  \eta(y,z)  u(y)u(z)  \, \mathrm{d}z  \, \mathrm{d}y \,.
\end{equation}
In particular, with $\varphi(y)=y$ we obtain from \eqref{mon_pres} that a solution $(v,u)$ to \eqref{eqv}-\eqref{AW} satisfies (formally) the monomer balance law
\begin{equation}\label{monomererhaltend}
\begin{split}
v(t) + &\int_{y_0}^{\infty}y u(t,y)\mathrm{d}y- v^0 - \int_{y_0}^{\infty}y u^0(y)\mathrm{d}y \\ 
&= \lambda t - \gamma \int_{0}^{t}v(s)\mathrm{d}s - \int_{0}^{t}\int_{y_0}^{\infty}y \mu(y) u(s,y)\mathrm{d}y\mathrm{d}s
\end{split}
\end{equation}
at time $t$. Thus, the number of monomers only changes due to natural production or metabolic degradation. This relation turns out to be crucial with respect to the existence of global  solutions as it provides suitable a priori estimates.  This, however, seems to be the only available  information.

In the following we use  $L_1(Y,y\mathrm{d}
y)$ as a state space for the population density $u$ and denote its positive cone by $L_1^+(Y,y\mathrm{d}y)$. This allows us to keep track of the  biologically important
quantities 
$$
\int_{y_0}^{\infty}  u(t,y)  \mathrm{d}y\qquad\text{and}\qquad \int_{y_0}^{\infty}  u(t,y) y \mathrm{d}y
$$
of all $PrP^{Sc}$ polymers respectively $PrP^{C}$ monomers forming those polymers.

\subsection{Classical Solutions for Bounded Kernels}

We consider first bounded kernels  $\mu$, $\beta$, $\eta$, and $\tau$. More precisely, we  let 
\bqn\label{boundedkernels}
\mu, \beta\in L_\infty^+(Y)\,,\quad \eta \in BC^1(Y \times Y, \mathbb{R}^+)
\eqn
and 
\bqn\label{ta}
\tau \in BC^1(Y, \mathbb{R}^+)\,,\qquad \tau (y)\ge  \tau_0\,,\quad y\in (y_0,\infty)\,,
\eqn
 for some constant $\tau_0>0$. The boundedness \eqref{boundedkernels} of the kernels in particular imply that the operators $L$ and $Q$ are bounded  and linear, respectively, bilinear operators from $L_1$ into itself. Using this we can proof the existence
and uniqueness of global classical solutions:

\begin{theorem}\label{T1}
Suppose \eqref{lambda_gamma}-\eqref{mon_pres},  \eqref{etasym}, \eqref{boundedkernels}, and \eqref{ta}. 
Then, given any initial values $v^0>0$ and $u^0\in L_1^+(Y,y\mathrm{d}
y)$ with
$\partial_yu^0\in L_1(Y,y\mathrm{d}
y)$ and $u^0(y_0)=0$, there exists a
unique global classical solution $(v,u)$ to \eqref{eqv}-\eqref{AW}
such that $v\in C^1(\mathbb{R}^+)$ and $u\in
C^1(\mathbb{R}^+,L_1(Y,y\mathrm{d}
y))$ with $\partial_yu\in C(\mathbb{R}^+,L_1(Y,y\mathrm{d}
y))$. This solution is positive, that is, $v(t)>0$, $u(t)\in L_1^+(Y,y\mathrm{d}y)$ for $t\ge 0$, and it is monomer preserving, that is, it satisfies the balance law \eqref{monomererhaltend}.
\end{theorem}

To prove Theorem~\ref{T1} we shall write \eqref{eqv}-\eqref{AW} as a quasilinear hyperbolic Cauchy problem for $u$, where the nonlinear transport term generates an evolution operator in the phase space $L_1(Y,y\mathrm{d}y)$ and the linear part $L$ can be considered as a linear perturbation thereof. Owing to the bilinear operator $Q$ the Cauchy problem is,  in contrast to \cite{SW06, W}, no longer homogeneous. We shall see  in Section~\ref{Sec3} that the fixed point argument to solve this Cauchy problem thus becomes more involved and has to be performed twice (in different function spaces) to cope with the lacking regularization of the hyperbolic evolution operator and the nonlinearities stemming from polymer joining.

\subsection{Weak Solutions for Unbounded Kernels} \label{Sec2.2}

The assumptions \eqref{boundedkernels}, \eqref{ta} that the kernels are bounded seem to be rather strong from a biological point of view  since they exclude e.g. splitting rates as in \eqref{5}. In order to include unbounded kernels we   weaken the notion of a solution.\\

\begin{definition}\label{D1}
Given $v^0 >0$ and $u^0 \in L_1^+(Y,y\mathrm{d}y)$ we call a pair $(v,u)$ a \textit{(monomer preserving) global weak solution} to \eqref{eqv}-\eqref{AW} provided  the following conditions are satisfied:
\begin{itemize}
\item[\rm(i)] $v \in C^1(\mathbb{R}^+)$  is a non-negative solution to \eqref{eqv},
\item[\rm(ii)] $u \in L_{\infty,\mathrm{loc}} \big(\mathbb{R}^+, L_1^+ \big)$  is  a weak solution to \eqref{equ}, that is, it satisfies for all $t > 0$ 
\begin{align} \label{mu_beta_u}
&[(s,y) \mapsto \big( \mu(y) + \beta(y)\big) u(s,y)] \in L_1 \big( (0,t) \times Y \big) 
\\ \label{eta_u}
&[(s,y,z)\mapsto \eta(y,z)u(s,y)u(s,z)]\in L_1((0,t)\times Y\times Y)
\end{align}
and 
\begin{align*}
\int_{y_0}^{\infty}\varphi(y) & u(t,y)\mathrm{d}y - \int_{y_0}^{\infty}\varphi(y)u^0(y)\mathrm{d}y
\\ = \, & \int_{0}^{t} \frac{v(s)}{1 + \nu \| u(s)\|_{L_1(Y, y \mathrm{d}y)}} \int_{y_0}^{\infty}\varphi'(y) \tau(y) u(s,y) \mathrm{d}y \mathrm{d}s \\
&-\int_{0}^{t} \int_{y_0}^{\infty}   \varphi(y) \mu(y)  u(s,y)  \, \mathrm{d} y \mathrm{d}s\\
&
 + \int_{0}^{t} \int_{y_0}^{\infty}  u(s,y)  \beta(y)  \left(   -\varphi(y)  +  2 \int_{y_0}^{y}  \varphi(z)  \kappa(z,y)   \, \mathrm{d} z  \right)   \, \mathrm{d} y \mathrm{d}s
 \\
& + \int_{0}^{t}\int_{y_0}^{\infty}  \int_{y_0}^{\infty}  (\varphi(y+z) - \varphi(y) - \varphi(z))  \eta(y,z)  u(s,y)u(s,z)  \, \mathrm{d}z  \mathrm{d}y \mathrm{d}s  
\end{align*}
for any test function $\varphi \in W^1_\infty(Y)$,
\item[\rm(iii)] the balance law \eqref{monomererhaltend} holds.
\end{itemize} 
\end{definition}
 The weak formulation in (ii) above is obtained by testing \eqref{equ} against $\varphi$ and using  the identity (\ref{tilde}).
To prove the existence of a weak solution we do no longer need bounded kernels but rather impose certain growth conditions.
More precisely, we suppose that
\begin{equation}
\label{Bedingung_mu_beta}
\mu, \beta \in L^+_{\infty, \mathrm{loc}}(Y) 
\end{equation}
and
\begin{equation}
\label{tau_2}
\tau \in C([y_0, \infty)) \quad\text{with}\quad  \tau_0 \leq \tau(y) \leq \tau_*y \ , \quad y \geq y_0 \ , 
\end{equation}
for some constants $\tau_0, \tau_* >0$. The measurable function $\kappa$ is supposed to satisfy \eqref{bin_split},\eqref{mon_pres} and given any $R > y_0$ it holds that
 \begin{equation}
\label{Bedingung_esssup_E}
 \lim_{\delta \rightarrow 0} \sup_{E \subset (y_0, R) \atop |E|  \leq \delta} \esssup_{y \in (y_0,R)} \beta(y) \int_{y_0}^{y} \mathbf{1}_{E}(z)\kappa(z,y)\mathrm{d}z = 0 \, ,
\end{equation}
where $|E|$ denotes the Lebesgue measure of a measurable set $E\subset Y$. Condition~\eqref{Bedingung_esssup_E} is used later on to guarantee the uniform integrability of a sequence of approximative solutions. Furthermore, let there be $y_1 \in Y$and $\delta_1 > 0$ such that
\begin{equation}
\label{Bedingung_kappa}
\int_{y_1}^{y}\left( 1 -\dfrac{z}{y} \right )\kappa(z,y) \mathrm{d}z \geq \delta_1 \ , \quad y \geq 2y_1\ .  
\end{equation}
The polymer joining kernel $\eta$ shall be a continuous function $Y \times Y \rightarrow \mathbb{R}^+$  satisfying \eqref{etasym} and
\begin{equation}\label{Bedingung_eta}
\begin{split}
   \eta(y,z)  \leq K \big({ y^{\alpha}z^{\rho} + y^{\rho}z^{\alpha}}\big)\,\qquad (y,z) \in Y \times Y\,,
\end{split}
\end{equation}
for some constant $K \geq 1$ and a pair of numbers $(\alpha, \rho)$   with  
\begin{equation}\label{theta}
0  \leq \alpha \leq \rho \leq 1\,,\quad \theta := \alpha + \rho \in [0,2]\,,
\end{equation}
ensuring the integrability of $Q$. In case that $\theta \in (1,2]$  we additionally require that there are $B > 0$, $\zeta > \theta - 1 $, and $0 < a < 1$ such that
\begin{equation}\label{Bedingung_beta}  
\beta(y)  \geq { By^{\zeta}} \,,\qquad   2 \int_{y_0}^{y}   z \kappa(z,y)  \mathrm{d}  z   \leq  a y \, ,     \qquad y \in Y \,. 
\end{equation}
The   imposed conditions are similar as in \cite{SW06,W,LW07}. We remark that a class of examples for  $\kappa$ is obtained when of the form
\begin{equation*}
\kappa (z, y) = \frac {1}{y} k_0 \left(  \frac{z}{y} \right) \ , \quad   y>y_0 \ , \quad  0<z<y \ ,
\end{equation*} 
with a non-negative integrable function $k_0$ defined on $(0,1)$ satisfying
\begin{equation*}
k_0( y) = k_0(1- y) \ , \quad   y \in (0,1) \ , \quad \int_0^1 k_0(y) \, \mathrm{d} y = 1\,.
\end{equation*}
One then readily checks that the conditions \eqref{bin_split}, \eqref{mon_pres}, \eqref{Bedingung_esssup_E}, and \eqref{Bedingung_kappa} hold. In particular, for $k_0 \equiv 1$ one has
\begin{equation*}
\kappa(z,y) = \frac{1}{y} \, , \quad y > y_0 \, , \quad 0 < z < y \, ,
\end{equation*}
as considered in \cite{EPW06, GvDWW07, GPW06}.

To state our result on existence of weak solutions we shall use the notation $L_{1,\mathrm{w}} (Y, y\mathrm{d}y)$ for the space $L_{1} (Y, y\mathrm{d}y)$ endowed with its weak topology.

\begin{theorem}\label{Tweak}
Suppose \eqref{lambda_gamma}-\eqref{mon_pres}, \eqref{etasym}, and \eqref{Bedingung_mu_beta}-\eqref{theta}. If $\theta=\alpha+\rho \in (1,2]$, then also suppose~\eqref{Bedingung_beta}.
Let $(v_0,u_0)$ with $v^0 > 0$ and  $u^0 \in L_1^+(Y,y\mathrm{d}y)$. Then there exists a monomer preserving global weak solution $(v,u)$  in the sense of Definition~\ref{D1} with  
$u \in C(\mathbb{R}^+, L_{1,\mathrm{w}}(Y, y\mathrm{d}y))$. 
\end{theorem}

The construction of a monomer preserving global weak solution results from a compactness argument. For suitably truncated bounded kernels we  first obtain from Theorem~\ref{T1} a sequence  $((v_n,u_n))_{n\in\N}$ of global classical solutions. We then use the balance law \eqref{monomererhaltend} and the Dunford-Pettis Theorem to derive compactness of this sequence in the space $C\big([0,T], \mathbb{R} \times L_{1,\mathrm{w}}(Y, y \mathrm{d}y) \big)$ for any given $T>0$. Finally, we show that any cluster point of the sequence  $((v_n,u_n))_{n\in\N}$ represents a monomer preserving global weak solution.

The previous compactness argument providing the existence of weak solutions does obviously not lead to uniqueness of such a solution. However, one can give sufficient conditions for uniqueness of weak solutions and thus obtain a well-posedness result for weak solutions  \cite{LW2}. This requires additional integrability properties of weak solutions as stated in the following result:

\begin{proposition}\label{C1}
Let  the assumptions of Theorem~\ref{Tweak} with $\theta=\alpha+\rho\le 1$ hold. If  $u^0 \in L_1^+(Y,y^{\sigma}\mathrm{d}y)$ for some $\sigma  \ge  1$, then  $u \in L_{\infty,\mathrm{loc}}\big( \mathbb{R}^+, L_1(Y, y^\sigma \mathrm{d}y) \big)$. 
\end{proposition}

\section{Proof of Theorem~\ref{T1}}\label{Sec3}

This section is dedicated to the existence and uniqueness of global classical solutions for bounded kernels for which we invoke the theory of evolution operators. Throughout we suppose the assumptions stated in Theorem~\ref{T1}.  

\subsection{Preliminaries}\label{Sec3.1}

The boundedness \eqref{boundedkernels} of the kernels  implies that the operators $L$ and $Q$ are bounded  and linear, respectively, bilinear operators from $L_1(Y, y \mathrm{d}y)$ into itself. More precisely, putting
$$ 
E_0:=L_1(Y, y\mathrm{d}y)
$$
equipped with the norm $\|\cdot\|_{0}:=\|\cdot\|_{L_1(Y, y\mathrm{d}y)}$, denoting its positive cone by $E_0^+$, and setting
$$
E_1 := \{u  \in E_0 : \partial_y(\tau u) \in E_0, u(y_0) = 0\} 
$$
equipped with the norm (see \eqref{tau_2})
$$
\|u\|_{1} := \|u\|_{0} + \|\partial_y (\tau u)\|_{0}\,, \ u\in E_1\,,
$$
we readily obtain:

\begin{lemma} \label{LQ}
(a) The operator $L: E_0\rightarrow E_0$ is bounded and linear with
$$
\|L[u]\|_{0} \leq c_*\, \big(\|\mu\|_\infty+ \|\beta\|_\infty\big)\,\|u\|_{0}  \ , \quad u \in  E_0 \,.
$$
(b) For $j\in\{0,1\}$ the operator $Q : E_j \times E_0 \to E_j$  is bounded and bilinear with
$$
\|Q[u, w]\|_{j} \leq  c_*\,\|\eta\|_\infty\, \|u\|_{j}\, \|w\|_{0} \,,  \quad u \in  E_j \,,\ w \in E_0  \, .
$$
\end{lemma}

It is worthwhile pointing out the property of $Q[\cdot,w]$  mapping $E_j$ into itself for both $j=0$ {\it and} $j=1$ when $w\in E_0$ is fixed.  This property is crucial for the existence of classical solutions.\\

To set the stage for a fixed point formulation of \eqref{eqv}-\eqref{AW} we next focus on the polymerization term in \eqref{equ} and  recall that it is the generator of a positive evolution operator on $E_0$ with domain $E_1$.
For this we define a diffeomorphism $\Theta: Y \to (0,\infty)$ by virtue of
\begin{equation}
\label{Phi}
\Theta(y) := \int_{y_0}^{y} \frac{\mathrm{d} y'}{\tau(y')} \, , \quad y \in Y \, .
\end{equation}
Given $f \in E_0$ we put
\begin{equation}
\label{W}
(\mathcal{W}(t)f)(y) := \mathbf{1}_{[t,\infty)} \left( \Theta(y)  \right) \frac{\tau(\Theta^{-1}(\Theta(y) - t))}{\tau(y)} f(\Theta^{-1}(\Theta(y) - t)) \ , \quad y \in Y \ , \ t \geq 0 \ .
\end{equation}
It then follows from  \cite{W} that $\{\mathcal{W}(t);t \geq 0\}$ is a strongly continuous positive semigroup on $E_0$ with generator $-A$  given by
\begin{equation*}
Au := \partial_y(\tau u), \quad u \in E_1\,.
\end{equation*}
Moreover, putting $\tau_*:=\|\tau\|_\infty/y_0$ so that $\tau(y)\le \tau_* y$, $\, y\in Y$, the estimate
\begin{equation} \label{stable}
\|\mathcal{W}(t)\|_{\mathcal{L}(E_0)} \leq e^{\tau_* t} \ , \quad t \geq 0 \,,
\end{equation}
holds and shows that the semigroup is stable in the sense of \cite{Pazy}.
Given $T\in (0,1]$  and $R>1$ define $J_T := [0,T]$  and 
\begin{equation}\label{Definition_V(T,R)}
\mathcal{V}_{T,R}:=\{V \in C^1(J_T) \ ; \  R^{-1} \leq V(t) \leq \|V\|_{C^1 (J_T)} \leq R\} \, .
\end{equation}
Then we introduce for $V \in \mathcal{V}_{T,R}$ the operator
\begin{equation}
\label{Definition_tilde_A}
 \mathbb{A}_V (t)u :=  V(t) \partial_y(\tau u) - L[u] \ , \quad u \in E_1 \ , \quad t \in J_T \,,
\end{equation}
and recall that $L$ is a bounded operator on $E_0$. It was shown in \cite{SW06,W} analogously to \cite[\S 5]{Pazy} that the stability \eqref{stable} implies that the operator family $\{- {\mathbb{A}}_V (t) \}_{t \in [0,T]}$ generates an evolution operator  on $E_0$. More precisely:

\begin{proposition}\label{ES}
Let $ R >0, T_0 > 0 $, and $0< T \leq T_0$ be given. Then  $ \{- {\mathbb{A}}_V (t) \}_{t \in [0,T]} $  generates for each $V \in \mathcal{V}_{T,R}$ a unique evolution operator $\mathbb{U}_V(t,s)$, $0 \leq s \leq t \leq T,$ in $E_0$ enjoying properties $(E_1)-(E_5)$ in \cite[$\S$5]{Pazy}. Moreover, there is  $\omega_0 := \omega_0 (T_0, R) > 0$ such that
\begin{equation}\label{Abschaetzung_in_L(E0)}
\|\mathbb{U}_V(t,s)\|_{\mathcal{L}(E_0)} \leq e^{\omega_0 (t-s)} \ , \quad 0 \leq s \leq t \leq T \ , \quad V \in \mathcal{V}_{T,R} \ ,
\end{equation} 
and
\begin{equation}
\label{Abschaetzung_in_L(E1)}
\|\mathbb{U}_V(t,s)\|_{\mathcal{L}(E_1)} \leq \omega_0 \ , \quad 0 \leq s \leq t \leq T \ , \quad V \in \mathcal{V}_{T,R} \,
\end{equation} 
and if $V, W  \in \mathcal{V}_{T,R}$, then
\begin{equation}\label{Abschaetzung_in_L(E1,E0)}
\|\mathbb{U}_W(t,s)-\mathbb{U}_V(t,s)\|_{\mathcal{L}(E_1, E_0)} \leq \omega_0 (t-s)\|W-V\|_{C(J_T)} \ , \quad 0 \leq s \leq t \leq T \ .
\end{equation} 
\end{proposition}

 The $u$ component of a solution $(v,u)$ to \eqref{eqv}-\eqref{AW} can  then be expressed in the form
$$
u(t)=\mathbb{U}_{V_u}(t,0)u^0+\int_{0}^{t}\mathbb{U}_{V_u}(t,s)\, Q[u(s),u(s)]  \, \mathrm{d}s \,,
$$
where 
$$
V_u(t)=\frac{v(t)}{1+\nu\displaystyle\int_{y_0}^\infty z u(t,z)\,\mathrm{d}z}\, ,
$$
which can be regarded as a fixed point equation for $u$.
Let us point out that \eqref{Abschaetzung_in_L(E1,E0)} guarantees Lipschitz continuity of the evolution operator $\mathbb{U}_{V_u}$ with respect to $V_u$ only when being considered as an operator from $E_1$ to $E_0$ while semigroup theory requires the nonlinearity $Q$ to map into $E_1$ to guarantee time differentiability of the integral term. To cope with these somewhat antagonizing facts the fixed point argument has to be performed twice, once in $E_1$ to ensure time differentiability  with regard to classical solutions and once in $E_0$ to handle the quasilinear part $V_u$ of the problem. As pointed out before, the properties of $Q$ stated in part (b) of Lemma~\ref{LQ} are crucial in this respect as we shall see in the next section.

\subsection{Local Existence}

\noindent Let $v^0>0$ and $u^0\in E_1\cap E_0^+$ be given and let $S>0$ be such that
\bqn\label{S}
S^{-1}\le v^0\le S\,,\qquad \|u^0\|_{1}\le S\,.
\eqn
{ We put
$$
r(S):=  2S \nu\left[ (S+2S\|\beta\|_\infty+\lambda)\frac{\|\tau\|_\infty}{y_0}+ \|\mu\|_\infty+\|\beta\|_\infty\right]
$$
and} then introduce for $\delta \in \{0,\nu\}$ the complete metric space
\bqnn
\begin{split}
Z_T^{\delta} := \Big\{   u \in C(J_T, E^{+}_{0}) \ ; & \ [t\mapsto  \delta \|  u(t)\|_{0}]\in  C^1(J_T)\,, \,\|  u(t)\|_{0} \leq 2S\,,\\
& { \Big| \frac{\rd}{\rd t}  \delta \|  u(t)\|_{0}\Big| \le r(S) }\,,\, t\in J_T\,,\,
  u(0) =u^0\Big\}
\end{split}
\eqnn
 equipped with the metric
\begin{equation*}
d_{Z_T^{\delta}}(  u,   w) := \|  u-  w\|_{C(J_T,E_0)} + \delta \big\| \|   u\|_{0}- \|   w\|_{0}   \big\|_{C^1(J_T)}\,,\qquad   u\,,   w\in Z_T^{\delta}\,.
\end{equation*}
Note that 
$$
\delta \|  u(t)\|_{0} = \delta \int_{y_0}^\infty yu(t,y)\,\rd y\,,\quad t\in J_T\,,
$$
for $u\in Z_T^{\delta}$ and that this term vanishes for $\delta=0$.
Let $\bar u\in Z_T^\nu$ be fixed and put
\begin{equation}
 g(\bar u(t)):=2\int_{y_0}^{\infty} \bar u(t,y) \beta (y) \int_0^{y_0} z \kappa(z,y)  \, \mathrm{d}z \, \mathrm{d}y\,,
\end{equation}
and
\begin{equation}
 p(\bar u(t)):=\frac{1}{1+\nu \|\bar u(t)\|_{0}} \int_{y_0}^{\infty} \tau(y)  \bar u(t,y) \, \mathrm{d}y \,.
\end{equation}
Note that both $g(\bar u)$ and $p(\bar u)$ are non-negative functions.
Consequently, the function $v_{\bar u}\in C^1(J_T)$, given by
\begin{equation}\label{v}
\begin{split}
v_{\bar u}(t) := &\exp\left(-\gamma t -  \int_0^t p(\bar u(\sigma)) \, \mathrm{d} \sigma\right)v^0 \\
&\quad  + \int _{0}^{t} \exp\left(-\gamma (t-s) - \int_s^t p(\bar u(\sigma)) \, \mathrm{d} \sigma\right)(\lambda + g(\bar u(s))) \, \mathrm{d} s \ , \quad  t \in J_T \ ,  
\end{split}
\end{equation}
defines the unique solution to \eqref{eqv} with $v_{\bar u}(0)=v^0$, when $u$ therein is replaced by $\bar u$. We then introduce
\begin{equation}\label{V}
V_{\bar u}(t):=\frac{v_{\bar u}(t)}{1+\nu \| \bar u(t)\|_0}\,,\quad t\in J_T\,,\quad \bar u\in  Z_T^\nu\,.
\end{equation}
Owing to \eqref{mon_pres} and the assumptions on $\beta$ and $\tau$ we have
\bqn\label{g}
 g(\bar u(t)) \le \|\beta\|_\infty\,\|\bar u(t)\|_0\,,\qquad  p(\bar u(t))\le \frac{\|\tau\|_\infty}{y_0}\, \|\bar u(t)\|_0
\eqn
for $0\le t\le T$. { Therefore, since $\bar u\in  Z_T^\nu$ and
$$
V_{\bar u}'(t)= \frac{v_{\bar u}'(t)}{1+\nu \| \bar u(t)\|_0}-\frac{v_{\bar u}(t)}{(1+\nu \| \bar u(t)\|_0)^2}\, \nu\, \frac{\rd}{\rd t} \|\bar u(t)\|_0\,,\quad t\in J_T\,,
$$
it readily follows from} \eqref{v} and \eqref{eqv} that there exists a constant $R(S)>1$ independent of $T\in (0,1]$ (and $\bar u$) such that $V_{\bar u}\in \mathcal{V}_{T,R(S)}$ for $\bar u\in Z_T^\nu$. Moreover, since
$$
\big\vert p(\bar u_1(t))-p(\bar u_2(t))\big\vert \le \frac{\|\tau\|_\infty}{y_0}\, \|\bar u_1(t)-\bar u_2(t)\|_0 +  \frac{\nu \|\tau\|_\infty}{y_0}\, \|\bar u_1(t)\|_0\,\big\vert\|\bar u_1(t)\|_0-\|\bar u_2(t)\|_0\big\vert\,,
$$
formula \eqref{v}  implies that there is a constant $c(S)>0$ independent of $T\in (0,1]$ such that
\bqnn
|v_{\bar u_1}(t)-v_{\bar u_2}(t)| 
\leq  T\,c(S)\,\|\bar u_1 - \bar u_2\|_{C(J_T, E_0)} \, , \quad 0 \leq  t \leq T \, , \quad  \bar u_1 \, , \bar u_2  \in Z_T^\nu \,. 
\eqnn 
Therefore, from 
$$
\big|V_{\bar u_1}(t)-V_{\bar u_2}(t)\big| 
\leq |v_{\bar u_1}(t)-v_{\bar u_2}(t)| +  \nu\,v_{\bar u_2}(t)\big\vert \|\bar u_1(t)\|_0-\|\bar u_2(t)\|_0\big\vert
$$
we obtain, on the one hand,
\bqn\label{lipv}
\begin{split}
\big|V_{\bar u_1}(t)-V_{\bar u_2}(t)\big| 
\leq c(S)\,\|\bar u_1 - \bar u_2\|_{C(J_T, E_0)} 
\end{split} 
\eqn
 and, on the other hand since $t\mapsto \nu \|u_j(t)\|_0$ is differentiable,
\bqn\label{lipv2}
\begin{split}
\big|V_{\bar u_1}(t)-V_{\bar u_2}(t)\big| 
& \le  T\,c(S)\,\big(\|\bar u_1 - \bar u_2\|_{C(J_T, E_0)} 
+  \nu\,\big\|\|\bar u_1\|_0 - \|\bar u_2\|_0\big\|_{C^1(J_T)}\big)\\
& =T\,c(S)\, d_{Z_T^\nu}(\bar u_1,\bar u_2)
\end{split} 
\eqn 
for $0 \leq  t \leq T\le 1$ and $\bar u_1 \, , \bar u_2  \in Z_T^\nu$.
We then consider for fixed $\bar u\in Z_T^\nu$ and $\hat u\in Z_T^0$ the equation 
\begin{equation}\label{E1a}
  u' +\mathbb{A}_{V_{\bar u}}(t)u = Q[u, {\hat u}(t)] \ , \quad t \in J_T \,, \qquad  u(0)=u^0\,,
\end{equation}
where the operator
$$ -\mathbb{A}_{V_{\bar u}}(t)u =  -V_{\bar u}(t) \partial_y(\tau u) + L[u]  \, , \quad u \in E_1 \, , \quad t \in J_T \, , 
$$  
is meaningful since $V_{\bar u}\in \mathcal{V}_{T,R(S)}$ and thus generates an evolution operator on $E_0$ with properties as stated in Proposition~\ref{ES}. Note that,  for $\bar u\in Z_T^\nu$ and $\hat u\in Z_T^0$ still fixed, the right hand side of \eqref{E1a} is a bounded linear operator from $E_1$ into itself  with respect to $u$ according to Lemma~\ref{LQ} (b) which depends continuously on $t$.
Standard arguments (e.g. see \cite[$\S 5$]{Pazy}) then ensure that \eqref{E1a} has a unique classical solution  
\bqn\label{u}
u:=u(\bar u, \hat u)\in C(J_T, E_1) \cap C^1(J_T, E_0)\,.
\eqn
 To prove that this solution is non-negative we introduce for technical reasons the constant
$$
\omega:= \frac{4S}{y_0}\,\|\eta\|_\infty+\|\mu+\beta\|_\infty\,.
$$ 
We then observe that $u$ also solves the problem
\begin{equation}\label{positive}
 w' +({A}_{V_ {\bar u}}(t) + \omega) w =  H(t)[w] \, , \quad t\in J_T\,, \qquad w(0)=u^0\,,
\end{equation} 
where 
$$
-{A}_{V_ {\bar u}}(t) w :=  -V_{\bar u}(t) \partial_y(\tau w)\,,\quad w\in E_1\,,\quad t\in J_T\,,
$$ 
generates an evolution operator on $E_0$ according to Proposition~\ref{ES}
and the bounded operator $H(t) \in \mathcal{L}(E_0)$, given by
$$
H(t)[w]:=Q[w, \hat u(t)] +  L[w]+ \omega w \,,\quad w\in E_0\,,
$$
depends continuously on $t$ and satisfies 
\bqn\label{h}
H(t)[w] \in E_0^+\,,\quad w \in E_0^+
\eqn 
due to the choice of the constant $\omega$. Recall that the semigroup $\mathcal{W}$ on $E_0$ generated by $-\partial_y(\tau \cdot)$ is positive. Then clearly $-A_{V_ {\bar u}}(t) - \omega$ generates a positive semigroup on  $E_0$ for each $t\in J_T$ fixed. The construction of evolution operators (see \cite[Theorem 5.3.1]{Pazy} and \cite{W}) entails that the evolution operator generated by $-{A}_{V_ {\bar u}}-\omega$ is positive as well. This together with $u^0\in E_0^+$ and  \eqref{h} then easily yields that $u(t)\in E_0^+$ for $t\in J_T$ (see also the proof of \cite[Theorem 3.1]{SW06}). 

Keeping $\bar u\in Z_T^\nu$ still fixed we next show that the mapping
$
\hat u\mapsto \Lambda_{\bar u}[\hat u]:=u(\bar u,\hat u)
$ is a contraction on $Z_T^0$ for sufficiently small $T\in (0,1]$. For this note (setting $E^+_1:=E_0^+\cap E_1$) that 
$$
\Lambda_{\bar u}[\hat u]=u(\bar u,\hat u)\in C(J_T, E_1^+) \cap C^1(J_T, E_0)
$$ 
satisfies
\begin{equation}\label{E11}
  u' +\mathbb{A}_{V_{\bar u}}(t)u = Q[u, {\hat u}(t)]  \ , \quad t \in J_T \,, \qquad  u(0)=u^0\,,
\end{equation}
and can thus be written as
\begin{equation}\label{Lambda}
\Lambda_{\bar u}[\hat u](t)=\mathbb{U}_{V_ {\bar u}}(t,0)u^0+\int_{0}^{t}\mathbb{U}_{V_ {\bar u}}(t,s)\, Q[\Lambda_{\bar u}[\hat u](s), \hat u(s)] \, \mathrm{d}s \ , \quad  t \in J_T \,,
\end{equation}
where the evolution operator $\mathbb{U}_{V_ {\bar u}}(t,s)$ enjoys the properties stated in Proposition~\ref{ES} with $\omega_0:=\omega_0(1,R(S))$. Consequently, it follows from Proposition~\ref{ES} and Lemma~\ref{LQ} that
\begin{equation}\label{323}
\begin{split}
\|\Lambda_{\bar u}[\hat u](t)\|_{k}
& \leq  \|\mathbb{U}_{V_ {\bar u}}(t,0)\|_{\mathcal{L}(E_k)}\, \|u^0\|_{k}+\int_{0}^{t}\|\mathbb{U}_{V_ {\bar u}}(t,s)\|_{\mathcal{L}(E_k)}\,\| Q[\Lambda_{\bar u}[\hat u](s), \bar u(s)]\|_{k}\, \mathrm{d}s
\\ & 
\le e^{\omega_0 (T(1-k)+k)}\,\|u^0\|_k +2S\,c_*\,\|\eta\|_\infty\, e^{\omega_0}\int_0^t \|\Lambda_{\bar u}[\hat u](s)\|_k\,\mathrm{d}s
\end{split}
\end{equation}
for $k=0,1$ and $0\le t\le T\le 1$. Thus, taking $k=0$  in \eqref{323} and recalling \eqref{S}, Gronwall's lemma entails that 
\bqn\label{2S}
\|\Lambda_{\bar u}[\hat u](t)\|_{0}\le 2S\,,\quad 0\le t\le T\,,
\eqn
provided that  $T=T(S)\in (0,1]$ is chosen sufficiently small. This shows that $\Lambda_{\bar u}[\hat u]=u(\bar u,\hat u)\in Z_T^0$ for $\hat u\in Z_T^0$. Moreover, taking $k=1$  in \eqref{323} Gronwall's lemma also implies that 
\bqn\label{bound}
\|\Lambda_{\bar u}[\hat u](t)\|_{1}\le m(S)\,,\quad 0\le t\le T\,,
\eqn
for some constant $m(S)>0$. 
To show that the mapping $\Lambda_{\bar u} : Z_T^0\rightarrow Z_T^0$
 is contractive let $\hat u_1, \hat u_2\in Z_T^0$.
Then, \eqref{Lambda} implies for $0\le t\le T$,
\begin{equation*}
\begin{split}
\|\Lambda_{\bar u}[\hat u_1](t) - \Lambda_{\bar u}[\hat u_2](t)\|_{0} & \leq   \int_{0}^{t}\|\mathbb{U}_{V_ {\bar u}}(t,s)\|_{\mathcal{L}(E_0)} \,\| Q[\Lambda_{\bar u}[\hat u_1](s), \hat u_1(s)-\hat u_2(s)]\|_{0}\, \mathrm{d}s
\\ & \quad  + \int_{0}^{t}\|\mathbb{U}_{V_ {\bar u}}(t,s)\|_{\mathcal{L}(E_0)} \,\| Q[\Lambda_{\bar u}[\hat u_1](s)-\Lambda_{\bar u}[\hat u_2](s), \hat u_2(s)]\|_{0}\, \mathrm{d}s 
\end{split}
\end{equation*}
and hence Proposition~\ref{ES}, Lemma~\ref{LQ}, and \eqref{2S}  give
\begin{equation*}
\begin{split}
\|\Lambda_{\bar u}[\hat u_1](t) - \Lambda_{\bar u}[\hat u_2](t)\|_{0} & \leq  2S \,  c_*\,\|\eta\|_\infty\, e^{\omega_0 }\,T\,  \| \hat u_1-\hat u_2\|_{C(J_T,E_0)}\\ 
& \quad  + 2S\, c_*\,\|\eta\|_\infty\, e^{\omega_0} \int_0^t \|\Lambda_{\bar u}[\hat u_1](s) - \Lambda_{\bar u}[\hat u_2(s)\|_0\,\rd s\,.
\end{split}
\end{equation*}
Gronwall's lemma implies
\begin{equation*}
\|\Lambda_{\bar u}[\hat u_1](t) - \Lambda_{\bar u}[\hat u_2](t)\|_{0}  \leq T\, c(S)\, \| \hat u_1-\hat u_2\|_{C(J_T,E_0)}\,,\quad 0\le t\le T \,.
\end{equation*}
Consequently, for each $\bar u\in Z_T^\nu$ the mapping $\Lambda_{\bar u}$ defines a contraction on $Z_T^0$ when $T=T(S)\in (0,1]$ is chosen sufficiently small and $\Lambda_{\bar u}$ thus has a unique fixed point $\Gamma(\bar u)\in Z_T^0$. Recall that  $\Gamma(\bar u)$ belongs in addition to $C(J_T, E_1) \cap C^1(J_T, E_0)$ according to \eqref{u}.

We next study  the mapping
$
\Gamma=[\bar u\mapsto \Gamma(\bar u)]
$
and show that it is a contraction on $Z_T^\nu$  provided  $T=T(S)\in (0,1]$ is sufficiently small. The corresponding unique fixed point along with the corresponding solution to \eqref{eqv} will then represent the local solution to \eqref{eqv}-\eqref{AW}.  To this end note that \eqref{Lambda} reads for the fixed point $u=\Gamma(\bar u)$ of $\Lambda_{\bar u}$  (omitting the hat of $u$ for simplicity) as 
\begin{equation}\label{Lambda1}
u(t)=\mathbb{U}_{V_ {\bar u}}(t,0)u^0+\int_{0}^{t}\mathbb{U}_{V_ {\bar u}}(t,s)\, Q[u(s), u(s)] \, \mathrm{d}s \ , \quad  t \in J_T \,.
\end{equation}
Now, consider $\bar u_1, \bar u_2\in Z_T^\nu$ and put $u_1:=\Gamma(\bar u_1)$ and $u_2:=\Gamma(\bar u_2)$.
Then we infer from \eqref{Lambda1} for $t \in J_T$
\begin{equation*}
\begin{split}
\|u_1(t) -  u_2(t)\|_{0} & \leq  \|\mathbb{U}_{V_ {\bar u_1}}(t,0) -\mathbb{U}_{V_ {\bar u_2}}(t,0)\|_{\mathcal{L}(E_1, E_0)}\,\|u^0\|_{E_1} 
\\ 
& \quad  + \int_{0}^{t}\|\mathbb{U}_{V_ {\bar u_1}}(t,s) - \mathbb{U}_{V_ {\bar u_2}}(t,s)\|_{\mathcal{L}(E_1, E_0)}\, \| Q[u_1(s), u_1(s)]\|_{1}\, \mathrm{d}s\\ 
& \quad  + \int_{0}^{t}\|\mathbb{U}_{V_ {\bar u_2}}(t,s)\|_{\mathcal{L}(E_0)} \,\| Q[u_2(s), u_1(s)-u_2(s)]\|_{0}\, \mathrm{d}s
\\ & \quad  + \int_{0}^{t}\|\mathbb{U}_{V_ {\bar u_2}}(t,s)\|_{\mathcal{L}(E_0)} \,\| Q[u_1(s)-u_2(s), u_1(s)]\|_{0}\, \mathrm{d}s \,,
\end{split}
\end{equation*}
and hence, from Proposition~\ref{ES}, Lemma~\ref{LQ}, and \eqref{bound},
\begin{equation*}
\begin{split}
\|u_1(t) - u_2(t)\|_{0} & \leq  \omega_0\, T\, \|V_{\bar u_1}-V_{\bar u_2}\|_{C(J_T)}\,\|u^0\|_{E_1} 
\\ 
& \quad  +  2S\, c_*\,\|\eta\|_\infty\, m(S) \omega_0\, T \, \|V_{\bar u_1}-V_{\bar u_2}\|_{C(J_T)}
\\  
& \quad  + 4S\, c_*\,\|\eta\|_\infty\, e^{\omega_0} \int_0^t \|u_1(s)-u_2(s)\|_0\,\rd s\,.
\end{split}
\end{equation*}
Gronwall's lemma implies
\begin{equation*}
\|u_1(t) - u_2(t)\|_{0}  \leq T\, c(S)\,  \|V_{\bar u_1}-V_{\bar u_2}\|_{C(J_T)} \,,\quad 0\le t\le T \,,
\end{equation*}
and thus, from \eqref{lipv},
\begin{equation}\label{est1}
\|u_1(t) - u_2(t)\|_{0} \leq T\, c(S)\,   \| \bar u_1-\bar u_2\|_{C(J_T,E_0)} \,,\quad 0\le t\le T \,,
\end{equation}
for some constant $c(S)>0$. Next, \eqref{Lambda1} for $u=\Gamma(\bar u)$ with $\bar u\in Z_T^\nu$ can also be written (see also \eqref{E11}) as
\begin{equation}\label{equ1}
 u' +V_{\bar u}(t) \partial_y(\tau u) = Q[u(t), u(t)] +  L[u(t)]  \,,\quad 0\le t\le T \,.
\end{equation} 
We shall integrate this equation
with respect to $y\in (y_0,\infty)$. Note that  $u(t) \in E_1$ and  the  assumption \eqref{ta} on $\tau$ imply
\begin{equation}\label{tau}
\int_{y_0}^{\infty} y \partial_y(\tau(y) u(t,y)) \, \mathrm{d}y 
   =- \int_{y_0}^{\infty} \tau(y) u(t,y) \, \mathrm{d} y  \, .
\end{equation} 
Next, \eqref{mon_pres}, \eqref{L_Test}, and \eqref{tilde} entail that
\begin{equation}
\begin{split}\label{Q1}
\int_{y_0}^{\infty} y L[u(t)](y) \, \mathrm{d}y +\int_{y_0}^{\infty} y  Q[u(t), u(t)](y)  \, \mathrm{d}y=  -\int_{y_0}^{\infty} y \mu(y) u(t,y)\, \mathrm{d}y
 - g(u(t)) \, .
\end{split}
\end{equation} 
Consequently, we derive from \eqref{equ1}-\eqref{Q1} that
\begin{equation}\label{b0}
\begin{split}
\frac{\rd}{\rd t} \int_{y_0}^{\infty} y u(t,y) \, \mathrm{d}y  =\ & V_{\bar u}(t) \int_{y_0}^{\infty} \tau(y) u(t,y) \, \mathrm{d} y
-\int_{y_0}^{\infty} y \mu(y) u(t,y)\, \mathrm{d}y- g(u(t)) \,,
\end{split}
\end{equation} 
for $ 0\le t\le T$ and $u=\Gamma(\bar u)$ with $\bar u\in Z_T^\nu$.  In particular, \eqref{b0} warrants
$$
\Big| \frac{\rd}{\rd t} \nu\|  u(t)\|_{0}\Big| \le\nu \left(V_{\bar u}(t) \frac{\|\tau\|_\infty}{y_0}+\|\mu\|_\infty+\|\beta\|_\infty\right) \|u(t)\|_0 \le r(S) \,,\quad t\in J_T\,,
$$
since \eqref{S}, \eqref{v} , and \eqref{g}  imply
$$
V_{\bar u}(t)\le v_{\bar u}(t)\le S+2S\|\beta\|_\infty+\lambda\,,\quad t\in J_T\,,
$$
and hence $u=\Gamma(\bar u)\in Z_T^\nu$ for $\bar u\in Z_T^\nu$.
Now, consider again $\bar u_1, \bar u_2\in Z_T^\nu$ and put $u_1:=\Gamma(\bar u_1)$ and $u_2:=\Gamma(\bar u_2)$. We then deduce from \eqref{b0}
\begin{equation*}
\begin{split}
\left\vert\frac{\rd}{\rd t}\int_{y_0}^\infty y\big(u_1(t,y)-u_2(t,y)\big)\,\rd y\right\vert
 \le \ &
\big\vert V_{\bar u_1}(t)-V_{\bar u_2}(t)\big\vert\, \int_{y_0}^{\infty} \tau(y) u_1(t,y) \, \mathrm{d} y\\
&+ V_{\bar u_2}(t)\left\vert  \int_{y_0}^{\infty} \tau(y) \big(u_1(t,y)-u_2(t,y)\big) \, \mathrm{d} y\right\vert\\
&+\left\vert\int_{y_0}^{\infty} y \mu(y) \big(u_1(t,y)-u_2(t,y)\big)\, \mathrm{d}y\right\vert+ \big\vert g\big(u_1(t)-u_2(t)\big)\big\vert\,.
\end{split}
\end{equation*} 
Since the kernels are bounded we obtain
\begin{equation*}
\begin{split} 
\left\vert\frac{\rd}{\rd t}\int_{y_0}^\infty y\big(u_1(t,y)-u_2(t,y)\big)\,\rd y\right\vert
 \le \ &
 c(S)\, \big\vert V_{\bar u_1}(t)-V_{\bar u_2}(t)\big\vert +c(S)\,\|u_1(t)-u_2(t)\|_0\\ 
&  + \big\vert g\big(u_1(t)-u_2(t)\big)\big\vert
\,.
\end{split}
\end{equation*} 
Invoking \eqref{g}, \eqref{lipv2}, and \eqref{est1} we get
\begin{equation}
\begin{split}\label{L1b}
\left\vert\frac{\rd}{\rd t}\int_{y_0}^\infty y\big(u_1(t,y)-u_2(t,y)\big)\,\rd y\right\vert
 \le \ &
 T\,c(S)\, d_{Z_T^\nu}(\bar u_1,\bar u_2) \,,\quad 0\le t\le T\,.
\end{split}
\end{equation} 
Combining \eqref{est1} and \eqref{L1b} shows that
\begin{equation*}
d_{Z_T^\nu}\big(\Gamma(\bar u_1),\Gamma(\bar u_2)\big)
 \le 
 T\,c(S)\, d_{Z_T^\nu}(\bar u_1,\bar u_2)\,, \quad \bar u_1, \bar u_2\in Z_T^\nu\,,
\end{equation*}
 that is, $\bar u\mapsto \Gamma(\bar u)$ is a contraction on $Z_T^\nu$ provided that $T=T(S)\in (0,1]$ is chosen sufficiently small. The contraction mapping principle then yields a unique fixed point $u$ so that $(v_u,u)$ is the unique solution to \eqref{eqv}-\eqref{AW} on the interval $[0,T]$. Since the choice of $T=T(S)$ only depends on $S$ from \eqref{S}, the following statement is immediate:

\begin{proposition}\label{Prop2}
Given the assumptions of Theorem~\ref{T1}, there exists a unique maximal  solution $(v,u)$ to \eqref{eqv}-\eqref{AW} belonging to $C(J, \mathbb{R}^+ \times E_1^+) \cap C^1(J, \mathbb{R} \times E_0)$ on a maximal interval $J$ which is open in $\mathbb{R}^+$. If $t^+ := \sup J < \infty$, then
\begin{equation}\label{tplus}
\liminf_{{t \nearrow t^+}} v(t) = 0  \quad \textit{ or } \quad {\limsup_{t \nearrow t^+}} \,\big(v(t) + \|u(t)\|_{E_1}\big) = \infty \, .
\end{equation}
\end{proposition}

\noindent Let us point out that the solution $(v,u)$ satisfies
\begin{equation*}
  u' +\mathbb{A}_{V_u}(t)u = Q[u, u]  \ , \quad t \in J \,, \qquad  u(0)=u^0\,,
\end{equation*}
with $V_u$ being defined in \eqref{V}
and $u$ can thus be written as
\begin{equation*}
u(t)=\mathbb{U}_{V_u}(t,0)u^0+\int_{0}^{t}\mathbb{U}_{V_u}(t,s)\, Q[u(s), u(s)] \, \mathrm{d}s \ , \quad  t \in J \,.
\end{equation*}

\subsection{Global Existence}

We next show that \eqref{tplus} cannot occur and the solution provided by Proposition~\ref{Prop2} thus exists on $J=\R^+$. For this we note the monomer balance law 
\begin{equation}\label{bb1}
\dot v (t) + \frac{\mathrm{d}}{\mathrm{d}t} \int_{y_0}^{\infty} y u(t,y) \, \mathrm{d}y = \lambda - \gamma v(t) - \int_{y_0}^{\infty} y \mu(y) u(t,y) \, \mathrm{d}y \ , \quad t \in J\,,
\end{equation}
which  now readily follows from \eqref{b0} and \eqref{eqv}. This turns out to be crucial for global existence as it implies the a priori bound
\begin{equation}\label{v_0_u_0}
v(t) + \|u(t)\|_{E_0} \leq v^0 + \|u^0\|_{E_0} + \lambda t \,,\quad t \in J\, .
\end{equation}
We then argue by contradiction and suppose that $t^+ < \infty$. Recalling \eqref{eqv}   
 we derive from \eqref{g} and \eqref{v_0_u_0}
$$ 
 v'(t) \leq \lambda + g(u(t)) \leq \lambda + \|\beta\|_{\infty}\|u(t)\|_{E_0} \leq c(t^+)  \, ,
$$
while \eqref{v_0_u_0} also implies
$$
v'(t) \geq -\gamma v(t) -  \frac{v(t)}{1+\nu\|u(t)\|_0} \int_{y_0}^\infty \tau(y) u(t,y)\,\rd y \ge  -\gamma v(t) -  v(t) \|\tau\|_\infty \|u(t)\|_0 \geq   - c(t^+) 
$$
for $t\in J$.
Therefore,
\begin{equation}\label{v1}
\|v\|_{C^1(J)} \leq c(t^+) \ .
\end{equation}
Furthermore, \eqref{v} and \eqref{g}  along with \eqref{v_0_u_0} warrant
\begin{equation}\label{v2}
v(t) \geq  \exp\left(-\gamma t -  \int_0^t p(\bar u(\sigma)) \, \mathrm{d} \sigma\right)v^0 \ge \exp\left(-\gamma t^+ -  c(t^+)\right)v^0  > 0 
\end{equation}
for $t \in J $.
Consequently, there exists $R = R(c(t^+))>0$ such that for each $0<T< t^+$ we have  $v \in \mathcal{V}_{T, R}$. Hence, it follows from Proposition~\ref{ES} that
\begin{equation*}
\|\mathbb{U}_{V_u}(t,s)\|_{\mathcal{L}(E_1)} \leq c(t^+) \ , \quad 0 \leq s \leq t < t^+ \, .
\end{equation*} 
We next infer from Lemma~\ref{LQ} and \eqref{v_0_u_0} that
$$ 
\|Q[u(t), u(t)]\|_{1}  \leq c_*\,\|\eta\|_\infty \|u(t)\|_1\,\|u(t)\|_0\le \, c(t^+)\|u(t)\|_{1} \ , \quad t \in J \ .$$
Therefore,
\begin{equation}
\begin{split}
\|u(t)\|_{1} & \leq \|\mathbb{U}_{V_u}(t,0)\|_{\mathcal{L}(E_1)}\| u^0\|_{1} + \int_0^t \|\mathbb{U}_{V_u}(t,s)\|_{\mathcal{L}(E_1)}\|Q[u(s), u(s)]\|_{1} \, \mathrm{d}s 
\\ & \leq c(t^+)\|u^0\|_{1} + c(t^+)\int_0^t\|u(s)\|_{1}\, \mathrm{d}s \ , \quad t \in J \ 
\end{split}
\end{equation}
so that Gronwall's lemma ensures
\begin{equation}\label{u1}
\|u(t)\|_{1}  \leq c(t^+) \ , \quad t \in J \, .
\end{equation}
Consequently, \eqref{v1}, \eqref{v2}, and \eqref{u1} rule out the occurrence of \eqref{tplus} contradicting our assumption of a finite $t^+$, hence $t^+=\infty$. This completes the proof of Theorem~\ref{T1}.

\subsection{Finite Speed of Propagation}

For later purposes when dealing with weak solutions we consider compactly supported initial values and show that the support propagates with finite speed provided that large polymers do not join.

\begin{lemma}\label{supp}
Let the assumptions of Theorem~\ref{T1} hold and suppose there is $S_1 > 0$ such that
 $\eta(y,z) = 0$    for $(y, z) \in   Y \times Y$  with  $z + y > S_1$. Let $v^0>0$ and $u^0\in E_1\cap E_0^+$ with
$
\mathrm{supp} \, u^0  \subset  [y_0, S_0]
$ 
for some $S_0 > y_0$. Let $(v,u)$ be the corresponding solution to \eqref{eqv}-\eqref{AW} provided by Theorem~\ref{T1}. Then
$$
\mathrm{supp} \, u(t)  \subset  [y_0, S(t)] \ , \quad t\ge 0 \, ,
$$ 
where $S$ is the solution to the ode
$$
S'(t) = \dfrac{v(t)}{1+\nu \|u(t)\|_{E_0}}\tau(S)\,,\quad t>0\,,\qquad S(0) = \max\{S_0,S_1\}\,,
$$         
that is,
$$ 
S(t) = \phi^{-1}\left( \int_{0}^{t} \frac{v(s)}{1 + \nu \|u(s)\|_{E_0}} \, \mathrm{d} s \right) \quad \text{ with }\quad  \phi(r) := \int_{\max\{S_0,S_1\}}^{r} \frac{\mathrm{d}z}{\tau(z)}\,. $$ 
\end{lemma}

\begin{proof}
The proof follows along the lines of \cite[Lemma~2.4]{SW06}. Indeed, noticing that $S$ is well-defined on $\R^+$, since $\tau$ is a bounded and continuous function,  and defining $P \in C^1\big( \R^+, L_1(Y) \big)$ according to
$$
P(t,y) := \int_{y}^{\infty}  u(t,y')  \, \mathrm{d} y' \ ,    \quad y \in Y, ~ t \ge 0 \,,
$$
we note that
\begin{equation*}
\begin{split}
\dfrac{\partial}{\partial t} P(t,y) 
 & =  \int_{y}^{\infty}  \partial_t u (t, y') \, \mathrm{d} y'
\\ & = \frac{v(t)}{1 + \nu \|u(t)\|_{E_0}} \tau(y) u(t,y)  +  \int_{y}^{\infty}  L[u(t)] (y')  \, \mathrm{d} y'  
+\int_{y}^{\infty}   Q[u(t), u(t)](y') \, \mathrm{d} y' \, .
\end{split}
\end{equation*}
Since
\begin{align*}
 \int_{S(t)}^{\infty}  \int_{y}^{\infty}  Q[u(t),u(t)](y') \, \mathrm{d} y'  \, \mathrm{d} y   & =
\int_{S(t)}^{\infty} \int_{y \vee 2y_0}^{\infty}   \int_{y_0}^{y'-y_0}\eta(y'-z,z)u(t,y'-z)u(t,z)\, \mathrm{d}z \, \mathrm{d} y'  \, \mathrm{d} y
\\ & \quad - 2 \int_{S(t)}^{\infty} \int_{y}^{\infty} u(t,y') \int_{y_0}^{\infty}  \eta(z,y')u(t,z) \, \mathrm{d}z \, \mathrm{d} y'  \, \mathrm{d} y \,,
\end{align*}
we obtain from the assumption on $\eta$ that
$$
\int_{S(t)}^{\infty}  \int_{y}^{\infty}  Q[u(t),u(t)](y') \, \mathrm{d} y'  \, \mathrm{d} y  = 0\,,\quad t\ge 0\,.$$
Therefore, using the positivity of $u$, \eqref{int_kappa} (which is implied by \eqref{bin_split}, \eqref{mon_pres}), and the definition of $S$, we compute
\begin{equation*}
\begin{split}
\dfrac{\mathrm{d}}{\mathrm{d}t}    \int_{S(t)}^{\infty}  P(t,y) \, \mathrm{d} y & =  \int_{S(t)}^{\infty}  \dfrac{\partial}{\partial t}  P(t,y) \, \mathrm{d} y  -  S'(t)  P(t, S(t))
\\ & = \frac{-v(t)}{1+ \nu \|u(t)\|_{E_0}} \int_{S(t)}^{\infty} \tau(y) \, \partial_y P(t,y) \mathrm{d}y +  \int_{S(t)}^{\infty}  \int_{y}^{\infty}   L[u(t)]  (y') \, \mathrm{d} y' \, \mathrm{d} y \quad \quad \quad \qquad \qquad
\\ & \qquad - \frac{v(t)}{1+ \nu \|u(t)\|_{E_0}} \tau(S(t)) P(t,S(t))\\ 
 & =  \frac{v(t)}{1+ \nu \|u(t)\|_{E_0}}  \int_{S(t)}^{\infty} \tau'(y) P(t,y) \, \mathrm{d} y + \int_{S(t)}^{\infty}  \int_{y}^{\infty}   L[u(t)]  (y') \, \mathrm{d} y' \, \mathrm{d} y\\
&  \le  \|\tau'\|_{\infty} v(t) \int_{S(t)}^{\infty} P(t,y) \, \mathrm{d} y
 + 2 \int_{S(t)}^{\infty}  \int_{y}^{\infty}   \beta(y'')  u(t, y'')   \int_{y}^{y''}   \kappa(y', y'')  \, \mathrm{d} y' \, \mathrm{d} y'' \, \mathrm{d} y \qquad
\\  &  \leq  \|\tau'\|_{\infty} v(t) \int_{S(t)}^{\infty} P(t,y) \, \mathrm{d} y +  2 \|\beta\|_{\infty}  \int_{S(t)}^{\infty} P(t,y)  \, \mathrm{d} y \,.
\end{split}
\end{equation*}
Thus,  Gronwall's lemma along with 
$$
\int_{S(0)}^{\infty} P(0,y) \, \mathrm{d} y = 0
$$
implies
$$
\int_{S(t)}^{\infty} P(t,y) \, \mathrm{d} y = 0 \ , \quad t \ge 0
$$
guaranteeing that $u(t,\cdot)$ vanishes on the interval $(S(t), \infty)$ for each $t\ge 0$.
\end{proof}

\section{Proof of Theorem~\ref{Tweak}}\label{Sec4}

We shall prove Theorem~\ref{Tweak} for unbounded kernels and thus suppose the conditions sated therein. 
Recall that  $v^0 > 0$ and  $u^0 \in L_1^+(Y,y\mathrm{d}y)$. We fix an arbitrary $T>0$. 

We first construct a suitable bounded approximation of the unbounded kernels for which classical solutions exist according to Theorem~\ref{T1} and we show then that a cluster point exists that is a weak solution for the original unbounded kernels. This approach
follows along the lines of \cite{LW07} but requires extensions  particularly due to the polymer joining term. For this we borrow ideas from \cite{ELMP} (see also \cite{L00}) used on the coagulation-fragmentation equations.

\subsection{Approximation by Bounded Kernels}\label{approx}

Let us first observe that $u^0 \in L_1^+(Y, y\mathrm{d}y)$ implies that we can apply a refined version of the de la Vall\'ee-Poussin Theorem \cite{Hoan} guaranteeing the existence of non-negative, non-decreasing, and convex function $\Phi \in C^{\infty}(\mathbb{R}^+)$ with $\Phi(0)=0$ such that $\Phi'$ is concave and
\begin{equation}\label{lim_Phi'}
\lim_{r \rightarrow \infty} \Phi'(r) = \lim_{r \rightarrow \infty} \dfrac{\Phi(r)}{r} = \infty
\end{equation}
with
$$
\int_{y_0}^{\infty} \Phi(y)u^0(y)\mathrm{d}y < \infty \, .
$$
We may then choose a sequence $(u^0_n)_{n\in\N}$ of non-negative, smooth, and compactly supported functions such that
\begin{equation}
\label{Konvergenz_u_0_n}
u^0_n \rightarrow u^0 \quad \text{in} \quad L^+_1(Y) \qquad \text{and} \qquad \sup_{n \in \mathbb{N}}\int_{y_0}^\infty \Phi(y) u^0_n(y)\mathrm{d}y < \infty \, . 
\end{equation}   
Next, we use a mollifier argument to construct a sequence $(\tau_n)_{n\in\N}$ in $BUC^{\infty}([y_0,\infty))$ satisfying
\begin{equation}\label{tau2}
0 < \frac{\tau_0}{2} \leq \tau_n(y) \leq \tau_* y  \ , \quad y\ge y_0 \, ,
\end{equation} 
such that 
\begin{equation}\label{tau3}
\tau_n \to \tau\quad \text{ uniformly on compact subsets of }\ Y\,.
\end{equation} 
Moreover, we can choose a sequence $(\eta_n)_{n\in\N}$ in $BUC^{\infty}(Y \times Y)$ such that
\begin{equation}\label{Bedingung_eta_n}
\eta_n(y,z) = \eta_n(z,y) \leq K\big({ y^{\alpha}z^{\rho} + y^{\rho}z^{\alpha}}\big)\,,\quad y\,, z\in Y\,,
\end{equation}
with constants $K$, $\alpha$, and $\rho$ stemming from \eqref{Bedingung_eta} and  
\begin{equation}\label{Bedingung_eta_n2}
\eta_n(y,z) = 0\quad \text{for }\ (y,z)\ \text{with }\ y+z > R_n\,,\qquad 2y_0 <R_n \rightarrow \infty\,,
\end{equation}
and such that
\begin{equation}\label{Bedingung_eta_n3}
\eta_n \to \eta\quad \text{ uniformly on compact subsets of }\ Y\times Y\,.
\end{equation} 
For $n\in \N$ we put 
$$
S_n^0 := \sup\{y \in (y_0,\infty) \ : \ y \in \mathrm{supp}~u_n^0\} \ , 
$$
and
$$
H_n(T) := \phi_n^{-1} \left( \int_0^T \left(v_0 + \int_{y_0}^{\infty} y u_n^0(y) \mathrm{d}y + \lambda t \right)  \mathrm{d} t \right) 
\,,\qquad \phi_n(r):= \int_{\max\{S_n^0,R_n\}}^{r} \frac{\mathrm{d}z}{\tau_n(z)} \,$$
and then introduce
$$
\mathcal{S}_n(T) := \max\{\mathcal{S}_{n-1}(T), H_n(T), n\} \, , \quad n \ge 1\,,\qquad \mathcal{S}_{0}(T) := H_{0}(T) \ .
$$
Let $\mu_n := \mathbf{1}_{[y_0, \mathcal{S}_n(T)]}\mu$ and $\beta_n := \mathbf{1}_{[y_0, \mathcal{S}_n(T)]}\beta$ for $n\in \N$. Thus, Theorem~\ref{T1} ensures the existence of a global non-negative classical solution  
$$
(v_n,u_n) \in C^1 \big(\mathbb{R}^+, \mathbb{R} \times E_0 \big) \cap C \big(\mathbb{R}^+, \mathbb{R} \times E_1 \big)
$$ 
to \eqref{eqv}-\eqref{equ} when $(\tau,\mu, \beta, \eta, u^0)$ is replaced with $(\tau_n,\mu_n, \beta_n, \eta_n, u^0_n)$. Moreover, the construction of $\mathcal{S}_n$ together with Lemma~\ref{supp}, \eqref{v_0_u_0},  and $\beta>0$ imply
\begin{equation}\label{supp_u_n}
\mathrm{supp} \, u_n(t) \subset [y_0, \mathcal{S}_n(T)] =\mathrm{supp} \, \beta_n \ , \quad t \in [0,T] \, .
\end{equation}
From \eqref{monomererhaltend} and \eqref{Konvergenz_u_0_n} we have
\begin{equation}\label{u_n_in_E_0}
v_n(t) + \int_{y_0}^{\infty} y u_n(t,y) \mathrm{d}y + \int_{0}^{t}\int_{y_0}^{\infty} y \mu_n(y) u_n(s,y) \mathrm{d}y\mathrm{d}s\leq c(T)\,,\qquad t\in [0,T]\,,\quad n\in\N\,,
\end{equation}
where $c(T)$ is independent of  $n$. We shall use  in the following the notation
$$
L_n[u](y) :=  - (\mu_n(y) + \beta_n(y))  u(y)  +  2 \int_{y}^{\infty}  \beta_n(z)  \kappa(y,z)  u(z)    \mathrm{d}z\,,\quad y\in Y\,,
$$
and  
$$
Q_n[u](y) := \mathbf{1}_{[y>2y_0]}\int_{y_0}^{y-y_0}\eta_n(y-z,z)u(y-z)u(z)\mathrm{d}z - 2u(y)\int_{y_0}^{\infty}  \eta_n(z,y)u(z) \mathrm{d}z\,,\quad y\in Y\,.
$$
In order to deal with the bilinear polymer joining terms we adapt the ideas  from
\cite[Lemma~3.2]{ELMP} (on the coagulation-fragmentation equations) to our situation and derive some estimates on the moments
$$
M_{s,n}(t):=\int_{y_0}^{\infty} y^{s} u_n(t,y)  \mathrm{d} y\,,\quad t\in [0,T]\,,
$$
for $s > 0$ and $n\in \N$. Note that all moments are well-defined due to the compact support of $u_n(t,\cdot)$.

\begin{lemma}\label{Lemma_M2}
Let $\theta = \alpha + \rho \in (1,2]$ in \eqref{Bedingung_eta}  and recall that then \eqref{Bedingung_beta} is supposed to hold. There is a constant $C_M(T)$ independent of $n$ such that
\begin{equation} 
M_{2,n}(t) \leq  C_M(T)(1 + t^{-1/\zeta})  \,,   \qquad t \in [0,T]\,,\quad n\in\N \, .
\end{equation}
\end{lemma}

\begin{proof}
As pointed out the proof  follows along the lines of \cite[Lemma~3.2]{ELMP}. Note that owing to \eqref{Bedingung_eta_n} we have 
$$
\eta_n(y,z) \leq  2K\big( { y^{\theta}  +  z^{\theta}} \big) \ ,  \quad (y,z) \in Y \times Y\,,
$$
and so it follows form \eqref{tilde} and \eqref{u_n_in_E_0}  for $t \in [0,T]$
\begin{equation*}
\begin{split}
\int_{y_0}^{\infty}  & y^2 Q_n[u_n(t)](y)  \mathrm{d} y 
\\ & \leq 2 K  \int_{y_0}^{\infty}  \int_{y_0}^{\infty}   ({ y^{\theta}  +  z^{\theta}})  ((y+z)^2 - y^2 -z^2) u_n(t,y) u_n(t,z) \mathrm{d} z   \mathrm{d} y 
\\ & =  8 K    \int_{y_0}^{\infty}  \int_{y_0}^{\infty}   { y^{1+\theta}}  z  u_n(t,y) u_n(t,z)   \mathrm{d} z   \mathrm{d} y  \,,
\end{split}
\end{equation*}
hence
\begin{equation*}
\begin{split}\label{Qf}
\int_{y_0}^{\infty}  & y^2 Q_n[u_n(t)](y)  \mathrm{d} y  \leq   c(T)\, M_{1+\theta,n}(t) \,  .
\end{split}
\end{equation*}
In addition,  \eqref{L_Test}, \eqref{Bedingung_beta},  \eqref{supp_u_n}, and the positivity of $\mu_n$, $\beta_n$ and $u_n$ imply for  $t \in [0,T]$
\begin{equation*}
\begin{split}
- \int_{y_0}^{\infty}    y^2 L_n[u_n(t)](y)  \mathrm{d} y 
& = \int_{y_0}^{\infty}   y^2 \big( \mu_n(y) +\beta_n(y)\big)\, u_n(t,y)  \mathrm{d} y  
\\ & \qquad - 2  \int_{y_0}^{\infty}   \beta_n(y) u_n(t,y)    \int_{y_0}^y  z^2 \kappa(z,y) \mathrm{d}  z \mathrm{d} y 
 \\ & \geq  (1-a)  B \int_{y_0}^{\infty}  { y^{2 +\zeta}}  u_n(t,y)  \mathrm{d} y \,,
\\
&=  (1-a) B  \, M_{2+\zeta,n}(t) \, . 
\end{split}
\end{equation*}
Next, using integration by parts we obtain from \eqref{tau_2}
\begin{equation*}
\begin{split}
\int_{y_0}^{\infty} y^2 \partial_y \big(\tau_n(y) u_n(t,y) \big) \mathrm{d}y 
 = y^2 \big( \tau_n(y) u_n(t,y) \big) \Big|_{y_0}^{\infty} - 2\int_{y_0}^{\infty}  y \tau_n(y) u_n(t,y) \mathrm{d}y 
\geq  -2\tau_* M_{2,n}(t)\,.
\end{split}
\end{equation*}
Therefore, integrating \eqref{equ} with respect to $y^2\rd y$  and using the above estimates and \eqref{u_n_in_E_0} we deduce
\begin{equation*}
\begin{split}
\frac{\mathrm{d}M_{2,n}(t)}{\mathrm{d}t} & =
   -  \frac{v_n(t)}{1+\nu\|u_n(t)\|_0} \int_{y_0}^{\infty} y^2 \partial_y \big( \tau_n(y) u_n(t,y) \big) \mathrm{d}y\\
&\quad 	+ \int_{y_0}^{\infty} y^2 L_n [u_n(t)](y) \mathrm{d}y
	  + \int_{y_0}^{\infty} y^2 Q_n [u_n(t)](y) \mathrm{d}y 
\\ & \leq  c(T) M_{2,n}(t) - (1-a)B \, M_{2+\zeta,n}(t) + c(T) M_{1+\theta,n}(t) \,,
\end{split} 
\end{equation*} 
hence, since $1+\theta > 2$,
\begin{equation}
\label{dM2}
\frac{\mathrm{d}M_{2,n}(t)}{\mathrm{d}t} + c_1 \, M_{2+\zeta,n}(t) \leq c(T)M_{1+\theta,n}(t) \, , \quad t \in [0,T] \, . 
\end{equation}
Next, since $\zeta > \theta - 1$, H\"older's inequality and the fact that  $M_{1,n}(t) \leq c(T)$ by \eqref{u_n_in_E_0} imply
$$
 M_{1+\theta,n}(t)
\le   M_{1,n}^{\frac{1+\zeta-\theta}{1+\zeta}}(t) M_{2+\zeta,n}^{\frac{\theta}{1+\zeta}}(t)
\leq c(T) \, M_{2+\zeta,n}^{\frac{\theta}{1+\zeta}}(t) 
$$
and plugging this into \eqref{dM2} and using Young's inequality (noticing that $\theta < 1+\zeta$) we derive
$$
\frac{\mathrm{d}M_{2,n}(t)}{\mathrm{d}t} + c_1(T) \, M_{2+\zeta,n}(t) \leq c(T) \, , \quad t \in [0,T] \, .
$$
Finally, using again H\"older's inequality and  \eqref{u_n_in_E_0} we get
$$
M_{2,n}(t) \leq M_{1,n}^{\frac{\zeta}{1+\zeta}}(t) M_{2+\zeta,n}^{\frac{1}{1+\zeta}}(t) \leq c(T)  M_{2+\zeta,n}^{\frac{1}{1+\zeta}} (t)
$$ 
and so
$$
\frac{\mathrm{d}M_{2,n}(t)}{\mathrm{d}t} + c_1(T)M_{2,n}^{1+\zeta}(t) \leq c(T)\,,\quad t\in [0,T] \, .
$$
The fact that the corresponding differential equation is solved by $t \mapsto C_M(1+t^{-\frac{1}{\zeta}})$ with $C_M$ only depending on $K, B, \theta, \zeta$, and $T$ yields the assertion.
\end{proof}

\begin{corollary}
\label{Lemma_Int_M2}
Let $\theta = \alpha + \rho \in (1,2]$ in \eqref{Bedingung_eta}  and assume \eqref{Bedingung_beta}. Then      
$$
 \int_{0}^{T} M_{\theta,n}(t) \mathrm{d} t \leq c(T) \,,\qquad t\in [0,T]\,,\quad n\in\N\,,
$$
with a constant $c(T)$ not depending on $n\in\N$.
\end{corollary}

\begin{proof}
Noticing that H\"older's inequality, \eqref{u_n_in_E_0}, and Lemma~\ref{Lemma_M2} imply
$$
  M_{\theta,n}(t) \leq M_{1,n}(t)^{2-\theta}\, M_{2,n}(t)^{\theta - 1}    \leq c(T)\, (1 + t^{-\frac{1}{\zeta}})^{\theta-1} 
\leq c(T)\, \left(1 + t^{-\frac{\theta-1}{\zeta}} \right)  
$$
for $t \in [0,T]$, the assertion follows since $\theta-1 < \zeta$.
\end{proof}

We next derive  a priori estimates which imply then later on the compactness of the sequence  $((v_n, u_n))_{n\in\N}$.

\begin{lemma}\label{L3}
There exists a constant $c(T)$ independent of $n$ such that
\begin{equation}\label{Phi_u_n}
\int_{y_0}^{\infty} \Phi(y) u_n(t,y) \mathrm{d}y \leq c(T) \, ,
\end{equation} 
\begin{equation}\label{I}
\int_{0}^{t} I_{1,n}(s) \mathrm{d}s +\int_{0}^{t} I_{2,n}(s) \mathrm{d}s \leq c(T) \, ,  
\end{equation} 
\begin{equation}\label{Phi_mu}
\int_{0}^{t} \int_{y_0}^{\infty} \Phi(y) \mu_n(y) u_n(s,y) \mathrm{d}y \mathrm{d}s \leq c(T) \, ,
\end{equation} 
for $t\in [0,T]$, where 
$$
I_{1,n}(s) := \int_{y_0}^{\infty} u_n(s,y) \beta_n(y) \int_{y_0}^{y} \left(  \dfrac{\Phi(y)}{y} - \dfrac{\Phi(z)}{z} \right) z \kappa(z,y) \mathrm{d}z \mathrm{d}y \, ,
$$
$$
I_{2,n}(s) := \int_{y_0}^{\infty} u_n(s,y) \beta_n(y) \dfrac{\Phi(y)}{y}  \int_{0}^{y_0}  z \kappa(z,y) \mathrm{d}z \mathrm{d}y \,.
$$
\end{lemma}

\begin{proof}
Recalling that $u_n(t,\cdot)$ is compactly supported we may test the corresponding equation \eqref{equ} with  $\Phi$ and obtain for $t \in [0,T]$ and $n\in \N$ on using \eqref{L_Test} and \eqref{tilde} that
\begin{align*}
\int_{y_0}^{\infty}   \Phi(y) u_n(t,y)\mathrm{d}y    &  = \int_{y_0}^{\infty} \Phi(y) u^0_n(y)\mathrm{d}y+  \int_{0}^{t}  \frac{v_n(s)}{1+\nu\| u_n(s)\|_0}  \int_{y_0}^{\infty}  \Phi'(y) \tau_n(y) u_n(s,y)  \mathrm{d}y  \mathrm{d}s  \nonumber
\\ & \qquad + \int_{0}^{t} \int_{y_0}^{\infty} \int_{y_0}^{\infty} \tilde{\Phi}(y,z) \eta_n(y,z) u_n(y) u_n(z) \mathrm{d}z \mathrm{d}y \mathrm{d}s \nonumber
\\& \qquad  -  \int_{0}^{t}  \int_{y_0}^{\infty}  \Phi(y)   (\mu_n(y) + \beta_n(y)) u_n(s,y)\mathrm{d}y  \mathrm{d}s  \nonumber
\\& \qquad + 2  \int_{0}^{t} \int_{y_0}^{\infty} u_n(s,y) \beta_n(y) \int_{y_0}^{y}  \Phi(z)  \kappa(z,y) \mathrm{d}z\mathrm{d}y  \mathrm{d}s  \,,
\end{align*}  
where $\tilde{\Phi}(y,z) := \Phi(y+z) - \Phi(y) - \Phi(z)$ for $y,z \in Y$.  
 We may rewrite the last two integrals on the right hand side  using \eqref{mon_pres} to get
\begin{align}\label{Abschaetzung_mit_Phi_I}
\int_{y_0}^{\infty}   \Phi(y) u_n(t,y)\mathrm{d}y  
& =  \int_{y_0}^{\infty} \Phi(y) u^0_n(y)\mathrm{d}y+\int_{0}^{t} \frac{v_n(s)}{1+\nu\|u_n(s)\|_0}  \int_{y_0}^{\infty}  \Phi'(y) \tau_n(y) u_n(s,y)  \mathrm{d}y  \mathrm{d}s  \nonumber
\\& \qquad  -  \int_{0}^{t}  \int_{y_0}^{\infty}  \Phi(y)   \mu_n(y)  u_n(s,y)\mathrm{d}y  \mathrm{d}s  \nonumber
\\& \qquad + \int_{0}^{t} \int_{y_0}^{\infty} \int_{y_0}^{\infty} \tilde{\Phi}(y,z) \eta_n(y,z) u_n(s,y) u_n(s,z) \mathrm{d}z \mathrm{d}y \mathrm{d}s            \nonumber
\\& \qquad  -  2\int_{0}^{t} ( I_{1,n}(s)  +  I_{2,n}(s)) \mathrm{d} s  \,.              
\end{align} 
We  then argue as in \cite[Section 4]{LW07}. Clearly, the terms involving $\mu_n$ and $I_{2,n}(s)$  are non-negative. The convexity of  $\Phi$ and  $\Phi(0) = 0$ imply that the mapping $y \mapsto \Phi(y)/y$ is non-decreasing so that $I_{1,n}(s)$ is non-negative as well. On the other hand, the convexity of $\Phi'$ along with $\Phi'(0) \geq 0$  entails 
$$
-\Phi'(y) \leq \Phi'(0) -  \Phi'(y) \leq -y\Phi''(y)
$$
and integrating this inequality yields  $y\Phi'(y) \leq 2\Phi(y)$ for $y \in Y$. Hence, since  $\Phi' \geq 0$, we obtain from \eqref{tau_2}  and  \eqref{u_n_in_E_0} 
\begin{align}\label{Phi'}
\int_{0}^{t} \frac{v_n(s)}{1+\nu\|u_n(s)\|_0}  \int_{y_0}^{\infty}  \Phi'(y)  \tau_n(y)  u_n(s,y)  \mathrm{d}y  \mathrm{d}s 
& \leq c(T) \int_{0}^{t} \int_{y_0}^{\infty} y \Phi'(y) u_n(s,y) \mathrm{d}y \mathrm{d}s  \nonumber
\\ & \leq c(T) \int_{0}^{t} \int_{y_0}^{\infty} \Phi(y) u_n(s,y) \mathrm{d}y \mathrm{d}s\,.
\end{align}
To bound the integral term involving $\eta_n$  in \eqref{Abschaetzung_mit_Phi_I} we argue along the lines of \cite[Proposition 3.4]{ELMP}. As therein we first note that (since $\Phi$  is convex and non-decreasing together with \cite[Lemma A.2]{L}) 
\bqn\label{316}
0 \leq \tilde{\Phi}(y,z) \leq 2 \dfrac{z \Phi(y) + y\Phi(z)}{y+z} \ , \qquad (y,z) \in Y \times Y\,,
\eqn
which shows that the integral term involving $\eta_n$ is non-negative. Introducing
$$
\Psi(y,z) := \tilde{\Phi}(y,z){ y^{\alpha}z^{\rho}} \ , \quad (y,z) \in Y \times Y \,,
$$
we also obtain from \eqref{316} 
\begin{align*}
 { \Psi(y,z)    \leq 2 (z \Phi(y) + y\Phi(z)) }\ , 
\end{align*}
if $\theta = \alpha + \rho \leq 1$.  If $\theta \in (1,2]$ and $y \geq z$, then \eqref{316} implies
\begin{equation*}
\begin{split}
 \Psi(y,z) &  \leq    { 2    \dfrac{y  z^{\theta} \Phi(y) + y y^{\theta} \Phi(z)} {y+z} }   \le 2 (z^{\theta} \Phi(y) + y^{\theta} \Phi(z))  
\end{split}
\end{equation*}
while the case $y \le z$ is analogous. We set $\theta_1 := \max\{1, \theta\}$ and obtain from \eqref{Bedingung_eta_n} and the estimates on $\Psi$
\begin{equation}\label{op}
\begin{split}
 \int_{0}^{t} \int_{y_0}^{\infty} & \int_{y_0}^{\infty} \tilde{\Phi}(y,z) \eta_n(y,z) u_n(s,y) u_n(s,z) \mathrm{d}z \mathrm{d}y \mathrm{d} s 
\\ & \leq  2 K \int_{0}^{t} \int_{y_0}^{\infty} \int_{y_0}^{\infty} \Psi (y,z)  u_n(s,y) u_n(s,z) \mathrm{d}z \mathrm{d}y  \mathrm{d} s 
  \\ & \leq {  8K } \int_{0}^{t} M_{\theta_1,n}(s) \int_{y_0}^{\infty} \Phi(y) u_n(s,y) \mathrm{d}y  \mathrm{d} s \, .
\end{split}
\end{equation}
 From \eqref{u_n_in_E_0} and Corollary~\ref{Lemma_Int_M2} we know that $ \int_{0}^{T} M_{\theta_1,n}(s) \mathrm{d} s$  is bounded independent of $n$. Therefore, the estimates \eqref{Konvergenz_u_0_n}, \eqref{Phi'}, and \eqref{op} allow us to apply  Gronwall's inequality to
\eqref{Abschaetzung_mit_Phi_I} in order to deduce that
$$
\int_{y_0}^{\infty} \Phi(y) u_n(t,y) \mathrm{d}y +
\int_{0}^{t} I_{1,n}(s) \mathrm{d}s +\int_{0}^{t} I_{2,n}(s) \mathrm{d}s +
\int_{0}^{t} \int_{y_0}^{\infty} \Phi(y) \mu_n(y) u_n(s,y) \mathrm{d}y \mathrm{d}s \leq c(T) \, ,
$$ 
whence the claim.
\end{proof}

\subsection{Compactness}  \label{cluster}

The estimates stated in Lemma~\ref{L3} allow us to show the weak compactness of the sequence  $((v_n,u_n))_{n\in\N}$ following \cite{ELMP,LW07}.

\begin{proposition}\label{P2}
There is a weakly compact subset $K_T$ of $L_1(Y, y\mathrm{d}y)$ such that $u_n(t) \in K_T$ for $n \in\N$ and $0 \leq t \leq T$. Moreover,
\begin{equation}
\label{beta_u_n}
\int_{0}^{T} \int_{y_0}^{\infty} \beta_n(y) u_n(s,y)  \mathrm{d}y \mathrm{d}s \leq c(T)\,,\quad n\in\N\,,
\end{equation}
for some positive constant $c(T)$ independent of $n \in\N$.
\end{proposition}

\begin{proof}
 Given $n \in \mathbb{N}$ and $t \in [0,T]$ it follows exactly as in \cite[Lemma 4.1]{LW07} that the properties of $\Phi$ and \eqref{Phi_u_n} imply
\begin{equation}
\label{u_n_R_infty}
\lim_{R \rightarrow \infty}  \sup_{n  \in\N \atop  t \in [0,T]}  \int_{R}^{\infty}  u_n(t,y) y \mathrm{d} y  = 0   
\end{equation}
and, for $S>R> 2 y_0$ fixed,
\begin{equation}
\label{Phi(S)_I}
\int_{S}^{\infty}  u_n(s,y) \beta_n(y)  \int_{y_0}^{R}  \kappa(z,y)  \mathrm{d} z  \mathrm{d} y  \leq \dfrac{1}{y_0\Phi(S)/S - \Phi(R)} I_{1,n}(s)\,,\quad s\in [0,T]\,.
\end{equation}
Given $\delta > 0$ we next define 
$$
\mathcal{E}_\delta^{n,R}(t):=\sup\,\left\{\int_{E} u_n(t,y)\mathrm{d}y \,;\, E \subset (y_0, R) \text{ measurable},  |E| \leq \delta\right\}
$$
and show that
\begin{equation}\label{delta1}
\lim_{\delta \rightarrow 0}  \sup_{n \in\N  \atop  t \in [0,T]} \mathcal{E}_\delta^{n,R}(t)=0\,.
\eqn   
Introducing
with $\mathbb{U}_{n}(t,s), 0 \leq s \leq t \leq \infty$, the (positive) evolution operator on $L_1(Y)$ corresponding to the operator $-A_{n}(t) :=  -V_{n}(t)\partial_y(\tau \cdot)$ with $V_{n}(t):=v_n(t)/(1+\nu\| u_n(t)\|_0)$ we first note that we can write $u_n$ in the form 
\begin{equation}\label{form}
u_n(t) = \mathbb{U}_{n}(t,0) u_n^0 + \int_0^t \mathbb{U}_{n}(t,s) (L_n[u_n(s)]  +  Q_n[u_n(s)]) \mathrm{d}s\,.
\end{equation}
Recall from \cite[Lemma 4.1]{W} that we have
\begin{equation}\label{E1b}
\sup_{E \subset (y_0, R), \atop  |E| \leq \delta} \int_{E} \mathbb{U}_{n}(t,s)f \, \mathrm{d}y \leq \sup_{F \subset (y_0, R), \atop |F| \leq \lambda_R^n(\delta)} \int_{F} f \, \mathrm{d}y
\end{equation}
for all $f \in L_1^+(Y)$, where
$$
\lambda_R^n(\delta) := \tau_*R \sup_{E \subset (y_0, R), \atop |E| \leq \delta} \int_{E} \frac{\mathrm{d}z}{\tau_n(z)} \,.
$$
Note that by \eqref{tau2} 
\begin{equation}\label{lambda_R}
\lambda_{R}^n(\delta) \leq \frac{2\tau_*R}{\tau_0} \delta  =: \lambda_R(\delta) \,.
\end{equation}
It then follows from \eqref{form}-\eqref{lambda_R} and the positivity of $u_n$ (i.e. neglecting negative contributions in \eqref{form}) that
\bqn\label{E3}
\begin{split}
\mathcal{E}_\delta^{n,R}(t) 
\,\leq\, \mathcal{E}_{\lambda_{R}(\delta)}^{n,R}(0)
&+ 2 \int_0^t \sup_{F \subset (y_0, R)  \atop |F| \leq \lambda_{R}(\delta)} \int_{y_0}^{\infty} u_n(s,y)\beta_n(y)\int_{y_0}^{y}  \mathbf{1}_{F}(z)\kappa(z,y)   \mathrm{d}z  \mathrm{d}y  \mathrm{d}s
\\
& +   \int_0^t \sup_{F \subset (2y_0, R)  \atop |F| \leq \lambda_{R}(\delta)}   \int_{F}\int_{y_0}^{y-y_0} \eta_n(y-z,z) u_n(s,y-z) u_n(s,z)  \mathrm{d}z  \mathrm{d}y  \mathrm{d}s \,.
\end{split}
\eqn
We now estimate the intergal terms on the right-hand side. First observe that, using \eqref{Phi(S)_I}, the second term on the right-hand side of \eqref{E3} can be bounded above as
\begin{equation} \label{hu}
2 \int_0^t \sup_{F \subset (y_0, R)  \atop |F| \leq \lambda_{R}(\delta)} \int_{y_0}^{\infty} u_n(s,y)\beta_n(y)\int_{y_0}^{y}  \mathbf{1}_{F}(z)\kappa(z,y)   \mathrm{d}z  \mathrm{d}y  \mathrm{d}s\le P(\delta, S)
\end{equation}
where
\begin{align*}
P(\delta, S) :=   2 \sup_{t \in [0,T] \atop n \in\N} \bigg\{ \int_0^t   \sup_{F  \subset (y_0, R)  \atop |F| \leq \lambda_R(\delta)}  
    \bigg( &\int_{y_0}^{S}  u_n(s,y) \beta_n(y)  \int_{y_0}^{R \land y}  \mathbf{1}_{F}(z) \kappa(z,y)   \mathrm{d}z  \mathrm{d}y  
\\   &    +    \dfrac{1}{ y_0 \Phi(S)/S   -  \Phi(R) }  I_{1,n}(s) \bigg)   \mathrm{d}s \bigg\} \,.
\end{align*}
As for the last term  on the right-hand side of \eqref{E3} we fix a measurable subset $F$ of $ (2y_0, R)$ with measure $|F| \leq \lambda_{R}(\delta)$ and $s\in [0,T]$. Then we deduce first from  \eqref{Bedingung_eta_n} and then from \eqref{u_n_in_E_0} along with the translation invariance of the Lebesgue measure that
\begin{equation*}
\begin{split}
\int_{F}\int_{y_0}^{y-y_0}  \eta_n(&y-z,z) u_n(s,y-z) u_n(s,z) \mathrm{d}z\mathrm{d}y
\\
&=  \int_{y_0}^{\infty} u_n(s,z) \int_{y_0}^{\infty} \mathbf{1}_{F}(y+z) \eta_n(y,z) u_n(s,y)  \mathrm{d}y\mathrm{d}z
\\
&  \leq c(R) \int_{y_0}^{\infty} u_n(s,z) \int_{y_0}^{\infty} \mathbf{1}_{-z+F}(y)  u_n(s,y)  \mathrm{d}y  \mathrm{d}z   
\\& 
\le  c(R,T)\, \mathcal{E}_{\lambda_{R}(\delta)}^{n,R}(s)\,.
\end{split}
\end{equation*}
Since clearly $\mathcal{E}_{\lambda_{R}(\delta)}^{n,R}(s)\le c(R) \,\mathcal{E}_{\delta}^{n,R}(s)$ due to the definition of $\lambda_{R}(\delta)$ we obtain
\begin{equation}\label{huu}
\begin{split}
 \int_0^t \sup_{F \subset (2y_0, R)  \atop |F| \leq \lambda_{R}(\delta)}   \int_{F}\int_{y_0}^{y-y_0} \eta_n(y-z,z) u_n(s,y-z) u_n(s,z)  \mathrm{d}z  \mathrm{d}y  \mathrm{d}s \le c(R,T) \int_0^t \mathcal{E}_{\delta}^{n,R}(s) \,\mathrm{d}s \,.
\end{split}
\end{equation}
Therefore, combining \eqref{E3}-\eqref{huu} we deduce that
\bqnn
\begin{split}
\mathcal{E}_\delta^{n,R}(t) 
\,\leq\, c(R)\,\mathcal{E}_{\delta}^{n,R}(0) + 
 P(\delta,S)+ c(R,T) \int_0^t \mathcal{E}_{\delta}^{n,R}(s)\, \mathrm{d}s
\end{split}
\eqnn
and hence, applying Gronwall's inequality observing that $\mathcal{E}_\delta^{n,R}\in C([0,T],\R^+)$,
\bqnn
\begin{split}
\mathcal{E}_\delta^{n,R}(t) 
\,\leq\, c(R,T)\,\big(\mathcal{E}_{\delta}^{n,R}(0) + 
 P(\delta,S)\big) \,,\qquad t\in [0,T]\,, \quad n\in\N\,,\quad \delta>0\,,
\end{split}
\eqnn
provided that $S>R>2y_0$. Noticing then
that, on the one hand, \eqref{Konvergenz_u_0_n} and the  Dunford-Pettis Theorem \cite[Theorem~4.21.2]{Edwards} imply
\bqnn
\lim_{\delta \rightarrow 0}\, \sup_{n \in\N}\, \mathcal{E}_{\delta}^{n,R}(0) =\lim_{\delta \rightarrow 0}\, \sup_{n \in\N}\, \sup_{F \subset (y_0, R)  \atop |F| \leq \delta}\int_{F} u^0_n(y)\mathrm{d}y = 0 \, ,
\eqnn
and, on the other hand, that  $P(\delta,S)$ can be made arbitrarily small  by choosing first $S>R$ large  and then $\delta$ small enough according to  \eqref{Bedingung_esssup_E}, \eqref{lim_Phi'}, \eqref{u_n_in_E_0}, and \eqref{I},  we deduce that
\bqn\label{DP1}
\lim_{\delta\to 0}\,\sup_{n\in\N}\, \mathcal{E}_\delta^{n,R}(t) =0\,,
\eqn
uniformly with respect to $t\in [0,T]$ and for every $R>y_0$. Combining \eqref{u_n_R_infty} and \eqref{DP1}, the existence of a weakly compact subset $K_T$ of $L_1(Y, y\mathrm{d}y)$ such that $u_n(t) \in K_T$ for $n \in\N$ and $0 \leq t \leq T$ is then a consequence of the Dunford-Pettis Theorem.

Finally, the estimate \eqref{beta_u_n}  is derived exactly as in  \cite[Lemma~4.1]{LW07} owing to assumptions \eqref{Bedingung_mu_beta}, \eqref{Bedingung_kappa},  the properties of $\Phi$ and the bounds \eqref{u_n_in_E_0}, \eqref{I}.
\end{proof}

\begin{lemma}\label{L4}
The family $\{u_n ; n \in\N\}$  is weakly equicontinuous in $L_1(Y, y \mathrm{d} y)$ at every  $t \in [0,T]$.
\end{lemma}

\begin{proof}
This follows along the lines of part (iii) of the proof of \cite[Theorem 4.3]{SW06} by using \eqref{u_n_in_E_0}, \eqref{Phi(S)_I}, and the weak compactness of $\{u_n(t)\,;\,t\in[0,T]\,,\, n\in\N\}$ in $L_1(Y,y\rd y)$ shown in Proposition~\ref{P2}.
\end{proof}

\begin{lemma}\label{L5}
The family $\{v_n ; n \in\N\}$ is relatively compact in $C([0,T])$.
\end{lemma}

\begin{proof}
This is a consequence of  Proposition~\ref{P2} and can be shown exactly as in \cite[Lemma 4.3]{LW07} by testing the truncated equation \eqref{equ} by $\varphi(y)=y$ and additionally observing that
$$\int_{y_0}^{\infty}  Q_n[u_n](y) y \mathrm{d} y = 0 \, .$$
\end{proof}

\subsection{Proof of Theorem~\ref{Tweak}}

We are now in a position to prove Theorem~\ref{Tweak}. It follows from Proposition~\ref{P2}, Lemma~\ref{L4}, Lemma~\ref{L5}, and a variant of the Arzel\`{a}-Ascoli Theorem \cite[Theorem~1.3.2]{Vrabie} that there are subsequences (not relabeled)
 $(v_n)$, $(u_n)$   and functions $v  \in C(\mathbb{R}^+)$, $u \in C \big( \mathbb{R}^+, L_{1,\mathrm{w}} (Y, y \mathrm{d}y)  \big) $ such that
\begin{align}
&v_n  \rightarrow v   \quad  \text{in}  \quad C([0,T]) \, ,\label{vv1}\\
&u_n  \rightarrow  u  \quad  \text{in}   \quad  C \big([0,T],  L_{1,\mathrm{w}} (Y, y \mathrm{d}y) \big) \label{uu1}
\end{align}
for each $T > 0$. In addition, $v(t) \geq 0$ and $u(t)  \geq  0$. It remains to show that $(v,u)$ is a weak solution to \eqref{eqv}-\eqref{equ}. Since $(v_n,u_n)$ satisfies the weak formulation given in  Definition~\ref{D1} we pass to the limit in each of the corresponding terms. This is rather standard by now and except for the bilinear polymer joining terms similar to \cite{LW07}. Indeed, using Fatou's Lemma we infer from   \eqref{beta_u_n} and \eqref{uu1} that
\begin{equation}
\label{beta_u}
[(t,y)  \mapsto  \beta(y)  u(t,y)]    \in   L_1\big( (0,T)  \times Y \big) \, ,
\end{equation} 
while \eqref{uu1} and \eqref{Bedingung_eta} clearly imply that
\begin{equation} \label{eta_infty}
[(t,y,z)  \mapsto \eta(y,z) u(t,z)  u(t,y)]\in   L_1\big( (0,T)  \times Y \times Y\big) \, . 
\end{equation}  
Also, \eqref{lim_Phi'}, \eqref{Phi_mu}, and \eqref{uu1}  ensure that
 \begin{equation}\label{lim_mu}
\lim_{n \rightarrow  \infty}  \int_{0}^{t}   \int_{y_0}^{\infty}    y \mu_n(y)   u_n(s,y)   \mathrm{d}  y    \mathrm{d}  s
=   \int_{0}^{t}   \int_{y_0}^{\infty}    y \mu(y)   u(s,y)   \mathrm{d}  y    \mathrm{d}  s   < \infty 
\end{equation}
for any fixed $t\in [0,T]$.
For $\varphi \in W_{\infty}^{1}(Y)$ it  follows then from \eqref{tau_2}, \eqref{tau2}, \eqref{tau3}, \eqref{u_n_R_infty}, \eqref{vv1}, \eqref{uu1} that
\begin{equation}\label{schwache_Konvergenz_2}
\begin{split}
\lim_{n \rightarrow \infty} \int_0^t \frac{v_n(s)}{1 + \nu \| u_n(s)\|_{0}}& \int_{y_0}^{\infty} \varphi'(y) \tau_n(y) u_n(s,y) \mathrm{d}y \mathrm{d}s \\
& =\int_0^t \frac{v(s)}{1 + \nu \| u(s)\|_{0}} \int_{y_0}^{\infty} \varphi'(y) \tau(y) u(s,y) \mathrm{d}y \mathrm{d}s \ .
\end{split}
\end{equation}
Therefore, using \eqref{int_kappa}, \eqref{Bedingung_mu_beta}, \eqref{lim_Phi'}, \eqref{I}, \eqref{Phi(S)_I}, \eqref{uu1}, and \eqref{beta_u}
it readily follows that
\begin{equation}\label{t6}
\begin{split}
\lim_{n \rightarrow \infty} \int_{0}^{t} \int_{y_0}^{\infty}  \varphi(y)  \mu_n(y)  u_n(s,y) \, \mathrm{d} y \mathrm{d}s=\int_{0}^{t} \int_{y_0}^{\infty}  \varphi(y)  \mu(y)  u(s,y) \, \mathrm{d} y \mathrm{d}s
\end{split}
\end{equation}
and
\begin{equation}\label{t7}
\begin{split}
\lim_{n \rightarrow \infty}\int_{0}^{t} \int_{y_0}^{\infty}  u_n(s,y)  &\beta_n(y)  \left(   -\varphi(y)  +  2 \int_{y_0}^{y}  \varphi(z)  \kappa(z,y)   \, \mathrm{d} z  \right)   \, \mathrm{d} y \mathrm{d}s\\
&
=\int_{0}^{t} \int_{y_0}^{\infty}  u(y)  \beta(y)  \left(   -\varphi(y)  +  2 \int_{y_0}^{y}  \varphi(z)  \kappa(z,y)   \, \mathrm{d} z  \right)   \, \mathrm{d} y \mathrm{d}s
\end{split}
\end{equation}
for any compactly supported test function  $\varphi \in W_{\infty}^{1}(Y)$, say with support $[y_0,R]$, by observing that $\mu_n=\mu$ and $\beta_n=\beta$ on $[y_0,R]$ when $n$ is so large that $\mathcal{S}_n(T)>R$ (see \eqref{supp_u_n}). For such a test function $\varphi$ one then also shows based on \eqref{Bedingung_eta}, \eqref{Bedingung_eta_n3}, \eqref{u_n_in_E_0}, \eqref{u_n_R_infty}, and \eqref{uu1} that
\begin{equation}\label{t8}
\begin{split}
\lim_{n \rightarrow \infty} \int_{0}^{t} \int_{y_0}^{\infty} \int_{y_0}^{\infty}  \varphi(y+z) & \eta_n(y,z) u_n(s,z) u_n(s,y) \, \mathrm{d} z\mathrm{d} y \mathrm{d}s\\
&=\int_{0}^{t} \int_{y_0}^{\infty} \int_{y_0}^{\infty}  \varphi(y+z)  \eta(y,z) u(s,z) u(s,y) \, \mathrm{d} z\mathrm{d} y \mathrm{d}s
\end{split}
\end{equation}
and
\begin{equation}\label{t9}
\begin{split}
\lim_{n \rightarrow \infty} \int_{0}^{t} \int_{y_0}^{\infty} \int_{y_0}^{\infty}  \varphi(y) & \eta_n(y,z) u_n(s,z) u_n(s,y) \, \mathrm{d} z\mathrm{d} y \mathrm{d}s\\
&=\int_{0}^{t} \int_{y_0}^{\infty} \int_{y_0}^{\infty}  \varphi(y)  \eta(y,z) u(s,z) u(s,y) \, \mathrm{d} z\mathrm{d} y \mathrm{d}s\,.
\end{split}
\end{equation}
 A classical truncation argument along with \eqref{beta_u}-\eqref{lim_mu} then entails that \eqref{t6}-\eqref{t9} hold true for any test function  $\varphi \in W_{\infty}^{1}(Y)$. Consequently, $u$ satisfies the weak formulation and it similarly follows from \eqref{I},  \eqref{u_n_in_E_0}, \eqref{vv1}, \eqref{uu1} that $v$ satisfies equation \eqref{eqv} and $v(t)>0$, $t\in [0,T]$.

Finally, \eqref{vv1}, \eqref{uu1}, and \eqref{lim_mu} guarantee that \eqref{monomererhaltend} also holds for $(v,u)$. This proves Theorem~\ref{Tweak}.

\subsection{Proof of Proposition~\ref{C1}}

Let now $\theta=\alpha+\rho\le 1$ and suppose that $u^0 \in L_1^+(Y,y^{\sigma} \mathrm{d}y)$  for some $\sigma\ge 1$. Let $t \in [0,T]$. Since $u_n(t,\cdot)$ is compactly supported we may test \eqref{equ} by $\varphi(y)=y^\sigma$ and obtain from \eqref{L_Test} and \eqref{tilde}
\begin{equation*}
\begin{split}\label{Abschaetzung_mit_sigma}
\dfrac{\mathrm{d}}{\mathrm{d} t} \int_{y_0}^{\infty}   y^\sigma u_n(t,y)\mathrm{d}y \nonumber
& = \sigma  v_n(t)  \int_{y_0}^{\infty}  y^{\sigma-1} \tau_n(y) u_n(t,y)  \mathrm{d}y    
  -   \int_{y_0}^{\infty}  y^\sigma   \big(\mu_n(y) + \beta_n(y)\big) u_n(t,y)\mathrm{d}y  
 \\   & \quad + 2 \int_{y_0}^{\infty} u_n(s,y) \beta_n(y) \int_{y_0}^{y}  z^\sigma  \kappa(z,y) \mathrm{d}z\mathrm{d}y  
 \\ & \quad +   \int_{y_0}^{\infty} \int_{y_0}^{\infty} \big((y+z)^\sigma -  y^\sigma -  z^\sigma \big)   \eta_n(y,z)u_n(t,y)u_n(t,z) \mathrm{d}z \mathrm{d}y  \,.                                                              
\end{split}
\end{equation*} 
Note that \eqref{mon_pres} entails 
$$
2\int_{y_0}^{y} z^{\sigma} \kappa(z,y) \mathrm{d}z \leq y^{\sigma} \,,\quad y>y_0\,,
$$
while \eqref{Bedingung_eta_n} implies
$$
\big((y+z)^\sigma -  y^\sigma -  z^\sigma\big)\,\eta_n(y,z) \le c(\sigma)\, y^{\sigma-1} z\, \,\eta_n(y,z)\le c\, y^\sigma z\,,\quad y_0\le z\le y\,,
$$
so that, according to \eqref{u_n_in_E_0},
\begin{equation*}
\begin{split}
 & \int_{y_0}^{\infty}  \int_{y_0}^{\infty}  ((y+z)^\sigma -  y^\sigma -  z^\sigma)   \eta_n(y,z)u_n(t,y)u_n(t,z) \mathrm{d}z \mathrm{d}y 
 \leq c(T) \int_{y_0}^{\infty}  y^{\sigma} u_n(t,y) \mathrm{d}y \,.
\end{split}
\end{equation*}
Hence we derive from \eqref{tau2} that
\begin{align*}
\label{Abschaetzung_mit_sigma}
\dfrac{\mathrm{d}}{\mathrm{d} t} \int_{y_0}^{\infty}   y^\sigma u_n(t,y)\mathrm{d}y 
 \leq c(T) \int_{y_0}^{\infty}  y^{\sigma} u_n(t,y) \mathrm{d}y \ , 
\end{align*} 
and consequently
\bqnn
\label{Norm_u_n_sigma}
\|u_n(t)\|_{L_1(Y,y^\sigma \mathrm{d}y)} \leq c(T) \, , \qquad t \in [0,T] \,,\quad n\in\N\,.
\eqnn
Since this estimate is preserved for $u$ due to \eqref{uu1}, Proposition~\ref{C1} follows.




\end{document}